\documentclass[11pt]{article}

\usepackage{latexsym,mathrsfs}
\usepackage{amsmath,amssymb}
\usepackage{amsthm,enumerate,verbatim}
\usepackage{amsfonts}
\usepackage{graphicx}
\usepackage{algorithm}
\usepackage{algorithmic}
\usepackage{url}  

\newtheorem{remark}{Remark}
\newtheorem{theorem}{Theorem}

\title{Coordinate Descent Methods for \\ Symmetric Nonnegative Matrix Factorization}
\author{Arnaud Vandaele\\
        University of Mons, Belgium \\
        {\tt arnaud.vandaele@umons.ac.be}
        \and
    		Nicolas Gillis\\
        University of Mons, Belgium \\
        {\tt nicolas.gillis@umons.ac.be}
        \and
    Qi Lei\\
        University of Texas, Austin \\
        {\tt leiqi@ices.utexas.edu}
        \and
    Kai Zhong\\
        University of Texas, Austin \\
        {\tt zhongkai@ices.utexas.edu}
        \and
    Inderjit Dhillon\\
        University of Texas, Austin \\
        {\tt inderjit@cs.utexas.edu}
        }
\date{}

\begin{document}

\maketitle

\begin{abstract} 
Given a symmetric nonnegative matrix $A$, symmetric nonnegative matrix factorization (symNMF) is the problem of finding a nonnegative matrix $H$, usually with much fewer columns than $A$, such that $A \approx HH^T$.  SymNMF can be used for data analysis and in particular for various clustering tasks. In this paper, we propose simple and very efficient coordinate descent schemes to solve this problem, and that can handle large and sparse input matrices. 
The effectiveness of our methods is illustrated on synthetic and real-world data sets, and we show that they perform favorably compared to recent state-of-the-art methods. 
\end{abstract}


\section{Introduction}
%
%
%
%
Nonnegative matrix factorization (NMF) has become a standard technique in data mining by providing low-rank decompositions of nonnegative matrices:
given a nonnegative matrix $X\in\mathbb{R}^{m\times n}_+$ and an integer $r < \min(m,n)$, the problem is to find  $W\in\mathbb{R}^{m\times r}_+$ and $H\in\mathbb{R}^{n \times r}_+$ such that 
$X\approx WH^T$. 
In many applications, the nonnegativity constraints lead to a sparse and part-based representation, 
and a better interpretability of the factors, e.g., when analyzing images or documents~\cite{LS99}. 

In this paper, we work on a special case of NMF where the input matrix is a symmetric matrix~$A$.
Usually, the matrix $A$ will be a \textit{similarity matrix} where the $(i,j)$th entry is a measure of the similarity between the $i$th and the $j$th data points. 
This is a rather general framework, and the user can decide how to generate the matrix $A$ from his data set by selecting an appropriate metric to compare two data points. 
As opposed to NMF, we are interested in a symmetric approximation $HH^T$ with the factor $H$ being nonnegative--hence symNMF is an NMF variant with $W = H$.  
If the data points are grouped into clusters, 
 each rank-one factor $H(:,j)H(:,j)^T$ will ideally correspond to a cluster present in the data set. 
In fact, symNMF has been used successfully in many different settings and was proved to compete with standard clustering techniques 
such as normalized cut, spectral clustering, $k$-means and spherical $k$-means; 
see~\cite{zass2005unifying, long2007relational, chen2008non, yang2012clustering, kuang2012symmetric, kuang2013symnmf, yglcw13} and the references therein.  

SymNMF also has tight connections with completely positive matrices~\cite{BS03, kalofolias2012computing}, that is, matrices of the form $A = HH^T, H \geq 0$, which play an important role in combinatorial optimization~\cite{Bu09}. Note that the smallest $r$ such that such a factorization exists is called the cp-rank of $A$. 
The focus of this paper is to provide efficient methods to compute good symmetric and 
nonnegative low-rank approximations $HH^T$ with $H \geq 0$ of a given nonnegative symmetric matrix $A$. 

Let us describe our problem more formally. 
Given a $n$-by-$n$ symmetric nonnegative matrix $A$ and a factorization rank $r$, symNMF looks for an $n$-by-$r$ nonnegative matrix $H$ such that $A \approx HH^T$.
The error between $A$ and its approximation $HH^T$ can be measured in different ways but we focus in this paper on the Frobenius norm: 
\begin{equation} \label{symNMF}
\min_{H \geq 0} \; F(H) \equiv  \frac{1}{4}\left\| A - HH^T \right\|_F^2 \; , 
\end{equation}
which is arguably the most widely used in practice. Applying standard non-linear optimization schemes to~\eqref{symNMF}, one can only hope to obtain stationary points, since the objective function of \eqref{symNMF} is highly non-convex, and the problem is NP-hard~\cite{DG14}. 
For example, two such methods to find approximate solutions to~\eqref{symNMF} were proposed in~\cite{kuang2013symnmf}:  
\begin{enumerate}

\item The first method is a Newton-like algorithm which exploits some second-order information without the prohibitive cost of the full Newton method. Each iteration of the algorithm has a computational complexity of $O(n^3r)$ operations. 

\item The second algorithm is an adaptation of the alternating nonnegative least squares (ANLS) method 
for NMF~\cite{kim2008toward, kim2011fast} where the term $||W - H||_F^2$ penalizing the difference between the two factors in NMF is added to the objective function. That same idea was used in~\cite{Ho2008} where the author developed two methods to solve this penalized problem but without any available implementation or comparison.

\end{enumerate}

In this paper, we analyze coordinate descent (CD) schemes for~\eqref{symNMF}. Our motivation is that the most efficient methods for NMF are CD methods; see \cite{CP09b,li2009fastnmf,GG12,hsieh2011fast} and the references therein. 
The reason behind the success of CD methods for NMF is twofold: 
(i) the updates can be written in closed-form and are very cheap to compute, and 
(ii) the interaction between the variables is low because many variables are expected to be equal to zero at a stationary point~\cite{NG11}. \\

The paper is organized as follows. 
In section~\ref{sec2}, we focus on the rank-one problem and present the general framework to implement an exact CD method for symNMF. 
The main proposed algorithm is described in section~\ref{sec3}. 
Section~\ref{init} discusses initialization and convergence issues. 
Section~\ref{sec5} presents extensive numerical experiments on synthetic and real data sets, which shows that our CD methods perform competitively with recent state-of-the-art techniques for symNMF. 

\section{Exact coordinate descent methods for SymNMF} \label{sec2}

Exact coordinate descent (CD) techniques are among the most intuitive methods to solve optimization problems. 
At each iteration, all variables are fixed but one, and that variable is updated to its optimal value. 
The update of one variable at a time is often computationally cheap and easy to implement.
However little interest was given to these methods until recently when CD approaches were shown competitive for certain classes of problems; see~\cite{wright2015coordinate} for a recent survey. 
In fact, more and more applications are using CD approaches, especially in machine learning when dealing with large-scale problems.

Let us derive the exact cyclic CD method for symNMF. The approximation $HH^T$ of the input matrix $A$ can be written as the sum of $r$ rank-one symmetric matrices:
\begin{equation}
A \; \approx \; \sum_{k=1}^r H_{:,k}H_{:,k}^T,
\end{equation}
where $H_{:,k}$ is the $k$th column of $H$. If we assume that all columns of $H$ are known except for the $j$th, the problem comes down to approximate a residual symmetric matrix $R^{(j)}$ with a rank-one nonnegative symmetric matrix $H_{:,j}H_{:,j}^T$:
\begin{equation} \label{symNMF-residual}
\min_{H_{:,j} \geq 0} \; \left\| R^{(j)} - H_{:,j}H_{:,j}^T \right\|_F^2 \; ,  
\end{equation}
where 
\begin{equation} \label{residual}
R^{(j)} \; = \; A-\sum_{k=1,k\neq j}^rH_{:,k}H_{:,k}^T\; .
\end{equation} 
For this reason and to simplify the presentation, we only consider the rank-one subproblem in the following section~\ref{r1subsec}, before presenting on the overall procedure in section~\ref{allsubsec}.

\subsection{Rank-one Symmetric NMF} \label{r1subsec}

Given a $n$-by-$n$ symmetric matrix $P \in \mathbb{R}^{n\times n}$, let us consider the rank-one symNMF problem 
\begin{equation} \label{r1symNMF}
\min_{h \geq 0} \; f(h) \equiv \frac{1}{4} \left\| P - h h^T \right\|_F^2 \; , 
\end{equation}
where $h\in\mathbb{R}^n_+$. 
If all entries of $P$ are nonnegative, the problem can be solved for example with the truncated singular value decomposition; this follows from the Perron-Frobenius and Eckart-Young theorems. 
In our case, the residuals $R^{(j)}$ will in general have negative entries--see~\eqref{residual}--which makes the problem NP-hard in general~\cite{belachew2015solving}. 
The optimality conditions for~\eqref{r1symNMF} are given by 
\begin{equation} \label{KKTr1symNMF}
h \geq 0, \nabla f(h) \geq 0 \text{, and } h_i \, \nabla f(h)_i = 0 \text{ for all } i , 
\end{equation}
where $\nabla f(h)_i$ the $i$th component of the gradient $\nabla f(h)$. 
For any $1 \leq i \leq n$, the exact CD method consists in alternatively updating the variables in a cyclic way:  
$$
\text{for } i = 1, 2, \dots, n: \quad h_i\leftarrow h_i^+ , 
$$ 
where $h_i^+$ is the optimal value of $h_i$ in~\eqref{r1symNMF} when all other variables are fixed. 
Let us show how to compute $h_i^+$.  
We have:
\begin{equation} \label{gradient} 
\nabla f(h)_i 
= 
h_i^3+\underbrace{\left(\sum_{l=1,l\neq i}^n h_l^2-P_{ii}\right)}_{a_i}h_i\underbrace{-\sum_{l=1,l\neq i}h_l P_{li}}_{b_i}, 
\end{equation}
where
\begin{eqnarray}
a_i & = & \sum_{l=1,l\neq i}^n h_l^2-P_{ii} = \|h\|^2-h_i^2-P_{ii} \, , \text{ and }  \label{ai}\\
b_i & = & -\sum_{l=1,l\neq i}h_lP_{li} = h_iP_{ii} - h^TP_{:,i} \, . \label{bi}
\end{eqnarray} 
If all the variables but $h_i$ are fixed, 
by the complementary slackness condition~\eqref{KKTr1symNMF}, the optimal solution $h_i^+$ will be either $0$ or a solution of the equation $\nabla f(h)_i = 0$, that is, a root of $x^3 + a_i x + b_i$. 
Since the roots of a third-degree polynomial can be computed in closed form, 
it suffices to first compute these roots and then evaluate $f(h)$ at these roots in order to identify the optimal solution $h_i^+$.  
The algorithm based on Cardano's method (see for example \cite{cardano1968ars}) is described as Algorithm~\ref{alg:bestroot} and runs in $O(1)$ time. Therefore, given that $a_i$ and $b_i$ are known, $h_i^+$ can be computed in $O(1)$ operations.

\begin{algorithm} 
\caption{$x=BestPolynomialRoot(a,b)$}
\label{alg:bestroot} 
\begin{algorithmic}[1]

\STATE INPUT: $a \in \mathbb{R}, b \in \mathbb{R}$
\STATE OUTPUT: $\arg \min_{x}$ $\frac{x^4}{4}+\frac{ax^2}{2}+bx$ such that $x\geq 0$.
\STATE $\Delta = 4a^3+27b^2$
\STATE $d = \frac{1}{2}\left(-b+\sqrt{\frac{\Delta}{27}}\right)$
\IF{$\Delta\leq 0$}
\STATE $r = 2\sqrt[3]{|d|}$
\STATE $\theta = \frac{phase_{}angle(d)}{3}$ 
\STATE $z^{\ast}=0$, $y^{\ast}=0$
\FOR{$k=0:2$}
\STATE $z = r\cos\left(\theta + \frac{2k\pi}{3}\right)$
\IF{$z\geq 0$ \AND $\frac{z^4}{4}+a\frac{z^2}{2}+bz < y^{\ast}$}
\STATE $z^{\ast}=z$
\STATE $y^{\ast}=\frac{z^4}{4}+a\frac{z^2}{2}+bz$ 
\ENDIF
\ENDFOR
\STATE $x = z^{\ast}$
\ELSE
\STATE $z = \sqrt[3]{d}+\sqrt[3]{\frac{1}{2}\left(-b-\sqrt{\frac{\Delta}{27}}\right)}$ 
\IF{$z\geq 0$ \AND $\frac{z^4}{4}+a\frac{z^2}{2}+bz<0$}
\STATE $x=z$
\ELSE
\STATE $x=0$
\ENDIF
\ENDIF
\end{algorithmic}
\end{algorithm}

The only inputs of Algorithm~\ref{alg:bestroot} are the quantities \eqref{ai} and \eqref{bi}.
However, the variables in~\eqref{r1symNMF} are not independent. 
When $h_i$ is updated to $h_i^+$, the partial derivative of the other variables, that is, the entries of $\nabla f(h)$, must be updated. 
For $l \in \{i+1,...,n\}$, we update: 
\begin{equation} 
a_l \leftarrow a_l + (h_i^+)^2 - h_i^2 \quad \text{ and } \quad b_l \leftarrow b_l + P_{li} (h_i^+-h_i).  \label{ai+bi+}
\end{equation} 
This means that updating one variable will cost $O(n)$ operations due to the necessary run over the coordinates of $h$ for updating the gradient. (Note that we could also simply evaluate the $i$th entry of the gradient when updating $h_i$, which also requires $O(n)$ operations; see section~\ref{sec3}.)  
Algorithm~\ref{CDr1SymNMF} describes one iteration of CD applied on problem~\eqref{r1symNMF}. In other words, if one wants to find a stationary point of problem~\eqref{r1symNMF}, Algorithm~\ref{CDr1SymNMF} should be called until convergence, and this would correspond to applying a cyclic coordinate descent method to~\eqref{r1symNMF}. 
In lines \ref{line:alg2-debutprecomp}-\ref{line:alg2-finprecomp}, the quantities $a_i$'s and $b_i$'s are precomputed. 
Because of the product $h^T P_{:,i}$ needed for every $b_i$, it takes $O(n^2)$ time.
Then, from line \ref{line:alg2-debutfor} to line \ref{line:alg2-finfor}, 
Algorithm \ref{alg:bestroot} is called for every variable and is followed by the updates described by~\eqref{ai+bi+}. 
Finally, Algorithm~\ref{CDr1SymNMF} has a computational cost of $O(n^2)$ operations.  Note that we cannot expect a lower computational cost since computing the gradient (and in particular the product $Ph$) requires $O(n^2)$ operations. 

\begin{algorithm}
\caption{$h=rankoneCDSymNMF(P, h_0)$}
\label{CDr1SymNMF}
\begin{algorithmic}[1]
\STATE INPUT: $P \in \mathbb{R}^{n\times n}, h_0 \in \mathbb{R}^{n}$
\STATE OUTPUT: $h \in \mathbb{R}^{n}_+$
\STATE $h = h_0$
\FOR{$i=1:n$} \label{line:alg2-debutprecomp}
\STATE $a_i = \|h\|^2_2 - h_{i}^2 - P_{ii}$ 
\STATE $b_i = h_{i}P_{ii}-h^TP_{:,i}$ 
\ENDFOR \label{line:alg2-finprecomp}
\FOR{$i=1:n$} \label{line:alg2-debutfor}
\STATE $h_{i}^+ = BestPolynomialRoot(a_i,b_i)$ 
\FOR{$l > i$}
\STATE $a_l \leftarrow a_l + (h_{i}^+)^2-h_{i}^2$ 
\STATE $b_l \leftarrow b_l + P_{li}(h_{i}^+-h_{i})$ 
\ENDFOR
\STATE $h_{i}=h_{i}^+$
\ENDFOR \label{line:alg2-finfor}
\end{algorithmic}
\end{algorithm}

\subsection{First exact coordinate descent method for SymNMF} \label{allsubsec}

To tackle SymNMF~\eqref{symNMF}, 
we apply Algorithm~\ref{CDr1SymNMF} on every column of $H$ successively, that is, we apply Algorithm~\ref{CDr1SymNMF} with $h = H(:,j)$ and $P = R^{(j)}$ for $j=1,\dots,r$. 
The procedure is simple to describe, see Algorithm~\ref{generalCDSymNMF} which implements the exact cyclic CD method applied to SymNMF. 
\begin{algorithm} 
\caption{$H=generalCDSymNMF(A,H_0)$}
\label{generalCDSymNMF}
\begin{algorithmic}[1]
\STATE INPUT: $A \in \mathbb{R}^{n\times n}, H_0 \in \mathbb{R}^{n\times r}$
\STATE OUTPUT: $H \in \mathbb{R}^{n\times r}_+$
\STATE $H = H_0$
\STATE $R=A-HH^T$ \label{alg3-full}
\WHILE{stopping criterion not satisfied}
\FOR{$j=1:r$}
\STATE $R^{(j)} \leftarrow R + H_{:,j}H_{:,j}^T$ \label{alg3-residual1}
\STATE $H_{:,j} \leftarrow rankoneCDSymNMF(R^{(j)}, H_{:,j})$ \label{alg3-callalg2}
\STATE $R \leftarrow R^{(j)}-H_{:,j}H_{:,j}^T$ \label{alg3-residual2}
\ENDFOR
\ENDWHILE
\end{algorithmic}
\end{algorithm} 
One can easily check that Algorithm~\ref{generalCDSymNMF} requires $O(n^2r)$ operations to update the $nr$ entries of $H$ once: 
\begin{itemize}

\item In step \ref{alg3-full}, the full residual matrix $R = A - HH^T$ is precomputed where the product $HH^T$ requires $O(rn^2)$ operations. 

\item In step \ref{alg3-residual1}, the residual matrix $R^{(j)}$ can be computed using the fact that $R^{(j)} = R + H_{:,j} H_{:,j}^T$, which requires $O(n^2)$ operations.  

\item In step \ref{alg3-callalg2}, Algorithm~\ref{CDr1SymNMF} is called, and requires $O(n^2)$ operations. 

\item In step \ref{alg3-residual2}, the full residual matrix $R = R^{(j)} - H_{:,j} H_{:,j}^T$ is updated, which requires $O(n^2)$ operations.  
\end{itemize}

Algorithm~\ref{generalCDSymNMF} has some drawbacks. In particular, the heavy computation of the residual matrix $R$ is unpractical for large sparse matrices (see below). 
In the next sections, we show how to tackle these issues and propose a more efficient CD method for symNMF, applicable to large sparse matrices.

%
%
%

\section{Improved Implementation of Algorithm~\ref{generalCDSymNMF}}  \label{sec3}

The algorithm for symNMF developed in the previous section (Algorithm~\ref{generalCDSymNMF}) is unpractical when the input matrix $A$ is large and sparse; in the sense that although $A$ can be stored in memory, Algorithm~\ref{generalCDSymNMF} will run out of memory for $n$ large. 
In fact, the residual matrix $R$ with $n^2$ entries computed in step \ref{alg3-full} of Algorithm~\ref{generalCDSymNMF} is in general dense (for example if the entries of $H$ are initialized to some positive entries--see section~\ref{init}), even if $A$ is sparse. Sparse matrices usually have $O(n)$ non-zero entries and, when $n$ is large, it is unpractical to store $O(n^2)$ entries (this is for example typical for document data sets where $n$ is of the order of millions).  

In this section we re-implement Algorithm~\ref{generalCDSymNMF} in order to avoid the explicit computation of the residual matrix $R$; see Algorithm~\ref{alg:cyclicCD}. 
While Algorithm~\ref{generalCDSymNMF} runs in $O(r n^2)$ operations per iteration and requires $O(n^2)$ space in memory (whether or not $A$ is sparse),  
Algorithm~\ref{alg:cyclicCD} runs in $O(r\max(K,nr))$ operations per iteration and requires $O(\max(K,nr))$ space in memory, where $K$ is the number of non-zero entries of $A$. Hence, 
\begin{itemize}

\item When $A$ is dense, $K = O(n^2)$ and Algorithm~\ref{alg:cyclicCD} will have the same asymptotic computational cost of $O(r n^2)$ operations per iteration as Algorithm~\ref{generalCDSymNMF}. 
However, it performs better in practice because the exact number of operations is smaller. 

\item When $A$ is sparse, $K = O(n)$ and Algorithm~\ref{alg:cyclicCD} runs in $O(r^2 n)$ operations per iteration, 
which is significantly smaller than Algorithm~\ref{generalCDSymNMF} in $O(r n^2)$, so that it will be applicable to very large sparse matrices. 
In fact, in practice, $n$ can be of the order of millions while $r$ is usually smaller than a hundred. 
This will be illustrated in section~\ref{sec5} for some numerical experiments on text data sets. 

\end{itemize}

In the following, we first assume that $A$ is dense when accounting for the computational cost of Algorithm~\ref{alg:cyclicCD}. 
Then, we show that the computational cost is significantly reduced when $A$ is sparse. 
Since we want to avoid the computation of the residual~\eqref{residual}, 
reducing the problem into rank-one subproblems solved one after the other is not desirable. 
To evaluate the gradient of the objective function in~\eqref{symNMF} for the $(i,j)$th entry of $H$, 
we need to modify the expressions \eqref{ai} and \eqref{bi} by substituting $R^{(j)}$ with $A-\sum_{k=1,k\neq j}^r H_{:,k}H_{:,k}^T$. 
We have 
\[
\nabla_{H_{ij}} F(H) = \nabla_{H_{ij}} \left(\frac{1}{4} ||A-HH^T||_F^2\right) =  H_{ij}^3 + a_{ij} H_{ij} +  b_{ij} , 
\]
where  
\begin{eqnarray}
	a_{ij} & = & \|H_{i,:}\|^2 + \|H_{:,j}\|^2 - 2H_{ij}^2 - A_{ii}, \text{ and } \label{aij} \\
	b_{ij} & = & H_{i,:}(H^TH)_{:,j} - H_{:,j}^TA_{:,i} - H_{ij}^3 - H_{ij}a_{ij}. \label{bij}
\end{eqnarray} 
The quantities $a_{ij}$ and $b_{ij}$ will no longer be updated during the iterations as in Algorithm \ref{generalCDSymNMF}, but rather computed on the fly before each entry of $H$ is updated. The reason is twofold:
\begin{itemize}
	\item it avoids storing two $n$-by-$r$ matrices,  and  
	\item the updates of the $b_{ij}$'s, as done in \eqref{ai+bi+}, cannot be performed in $O(n)$ operations without the matrix $R^{(j)}$.
\end{itemize}
However, in order to minimize the computational cost, the following quantities will be precomputed and updated during the course of the iterations: 
\begin{itemize}

\item $\|H_{i,:}\|^2$ for all $i$ and $\|H_{:,j}\|^2$ for all $j$:  if the values of $\|H_{i,:}\|^2$ and $\|H_{:,j}\|^2$ are available, $a_{ij}$ can be computed in $O(1)$; see \eqref{aij}. 
Moreover, when $H_{ij}$ is updated to its optimal value $H_{ij}^+$, we only need to update $\|H_{i,:}\|^2$ and $\|H_{:,j}\|^2$ which can also be done in $O(1)$:
\begin{equation}
	\|H_{i,:}\|^2 \leftarrow \|H_{i,:}\|^2 + (H_{ij}^+)^2 - H_{ij}^2, \label{aij+}
\end{equation}
\begin{equation}
	\|H_{:,j}\|^2 \leftarrow \|H_{:,j}\|^2 + (H_{ij}^+)^2 - H_{ij}^2. \label{bij+}
\end{equation} 
Therefore, pre-computing the $\|H_{i,:}\|^2$'s and $\|H_{:,j}\|^2$'s, which require $O(rn)$ operations, allows us to compute the $a_{ij}$'s in $O(1)$. 

\item The $r$-by-$r$ matrix $H^TH$: by maintaining $H^TH$, computing $H_{i,:}(H^TH)_{:,j}$ requires $O(r)$ operations. 
Moreover, when the $(i,j)$th entry of $H$ is updated to $H_{ij}^+$, updating $H^TH$ requires $O(r)$ operations: 
\begin{multline} 
	(H^TH)_{jk} \leftarrow (H^TH)_{jk} - H_{ik}(H_{ij}^+-H_{ij}), \hspace{0.4cm} \\k=1,...,r. \label{HH}
\end{multline}
\end{itemize} 
To compute $b_{ij}$, we also need to perform the product $H_{:,j}^TA_{:,i}$; see \eqref{bij}.
This requires $O(n)$ operations, which cannot be avoided and is the most expensive part of the algorithm.

In summary, by precomputing the quantities $\|H_{i,:}\|^2$, $\|H_{:,j}\|^2$ and $H^TH$, it is possible to apply one iteration of CD over the $nr$ variables in $O(n^2r)$ operations.  
The computational cost is the same as in Algorithm \ref{generalCDSymNMF}, in the dense case, but no residual matrix is computed; 
see Algorithm \ref{alg:cyclicCD}. 
\begin{algorithm}
\caption{$H=cyclicCDSymNMF(A,H_0)$}
\label{alg:cyclicCD}
\begin{algorithmic}[1]
\STATE INPUT: $A \in \mathbb{R}^{n\times n}, H_0 \in \mathbb{R}^{n\times r}$
\STATE OUTPUT: $H \in \mathbb{R}^{n\times r}$
\STATE $H=H_0$
\FOR{$j=1:r$} \label{line:debutprecomp}
\STATE $C_j = \|H_{:,j}\|^2$
\ENDFOR
\FOR{$i=1:n$}
\STATE $L_i = \|H_{i,:}\|^2$
\ENDFOR
\STATE $D = H^TH$ \label{line:finprecomp} 
\WHILE{stopping criterion not satisfied}
\FOR{$j=1:r$} \label{line:j1r}
\FOR{$i=1:n$} \label{line:i1n}
\STATE $a_{ij} \leftarrow C_j + L_i - 2H_{ij}^2 - A_{ii}$
\STATE $b_{ij} \leftarrow H_{i,:}^T(D)_{j,:} - H_{:,j}^TA_{:,i} - H_{ij}^3 - H_{ij}a_{ij}$ \label{line:calculbij}
\STATE $H_{ij}^+ \leftarrow BestPolynomialRoot(a_{ij},b_{ij})$
\STATE $C_j \leftarrow C_j + (H_{ij}^+)^2-H_{ij}^2$
\STATE $L_i \leftarrow L_i + (H_{ij}^+)^2-H_{ij}^2$
\STATE $D_{j,:} \leftarrow D_{j,:} - H_{i,:}(H_{ij}^+-H_{ij})$
\STATE $D_{:,j} \leftarrow D_{j,:}$
\ENDFOR
\ENDFOR
\ENDWHILE
\end{algorithmic}
\end{algorithm}

From line \ref{line:debutprecomp} to line \ref{line:finprecomp}, the precomputations are performed in $O(nr^2)$ time where computing $H^TH$ is the most expensive part.  
Then the two loops iterate over all the entries to update each variable once.
Computing $b_{ij}$ (in line \ref{line:calculbij}) is the bottleneck of the CD scheme as it is the only part in the two loops which requires $O(n)$ time. 
However, when the matrix $A$ is sparse, the cost of computing $H_{:,j}^TA_{:,i}$ for all $i$, that is computing $H_{:,j}^T A$, drops to $O(K)$ where $K$ is the number of nonzero entries in $A$. 
Taking into account the term $H_{i,:} (H^T H)_{j,:}$ to compute $b_{ij}$ that requires $O(r)$ operations, we have that Algorithm \ref{alg:cyclicCD} requires $O(r \max(K,nr) )$ operations per iteration. 



\section{Initialization and Convergence}  \label{init}

In this section, we discuss initialization and convergence of Algorithm~\ref{alg:cyclicCD}. 
We also provide a small modification for Algorithm~\ref{alg:cyclicCD} to perform better (especially when random initialization is used).

\paragraph{Initialization}
 
In most previous works, the matrix $H$ is initialized randomly, using the uniform distribution in the interval [0,1] for each entry of $H$~\cite{kuang2013symnmf}. Note that, in practice, to obtain an unbiased initial point, the matrix $H$ should be multiplied by a constant $\beta^*$ such that 
\begin{multline} \label{scale}
\beta^* = \arg\min_{\beta \geq 0} || A - (\beta H_0) (\beta H_0)^T||_F \\
= \sqrt{ \frac{\langle A , H_0 H_0^T \rangle}{\langle H_0 H_0^T, H_0 H_0^T \rangle} }
= \sqrt{ \frac{\langle A H_0, H_0 \rangle}{||H_0^T H_0||_F^2} }. 
\end{multline} 
This allows the initial approximation $H_0 H_0^T$ to be well scaled compared to $A$. When using such an initialization, we observed that using random shuffling of the columns of $H$ before each iteration (that is, optimizing the columns of $H$ in a different order each time we run Algorithm~\ref{alg:cyclicCD}) performs in general much better; see section~\ref{sec5}. 


\begin{remark}[Other heuristics to accelerate coordinate descent methods]
During the course of our research, we have tried several heuristics to accelerate Algorithm~\ref{alg:cyclicCD}, including three of the most popular strategies: 
\begin{itemize}

\item \emph{Gauss-Southwell strategies}. We have updated the variables by ordering them according to some criterion (namely, the decrease of the objective function, and the magnitude of the corresponding entry of the gradient). 

\item \emph{Variable selection}. Instead of optimizing all variables at each step, we carefully selected a subset of the variables to optimize at each iteration (again using a criterion based on the decrease of the objective function or the magnitude of the corresponding entry of the gradient). 

\item \emph{Random shuffling}. We have shuffled randomly the order in which the variables are updated in each column. This strategy was shown to be superior in several context, although a theoretical understanding of this phenomenon remains elusive~\cite{wright2015coordinate}. 

\end{itemize}
However, these heuristics (and combinations of them) would not improve significantly the effectiveness of Algorithm~\ref{alg:cyclicCD} hence we do not present them here. 
\end{remark}

Random initialization might not seem very reasonable, especially for our CD scheme. In fact, at the first step of our CD method, the optimal values of the entries of the first column $H_{:,1}$ of $H$ are computed sequentially, trying to solve
\[
\min_{H_{:,1} \geq 0} || R^{(1)} -  H_{:,1} H_{:,1}^T||_F^2 \qquad \text{ with } R^{(1)} = A - \sum_{k = 2}^r H_{:,k} H_{:,k}^T. 
\]
Hence we are trying to approximate a matrix $R^{(1)}$ which is the difference between $A$ and a randomly generated matrix $\sum_{k = 2}^r H_{:,k} H_{:,k}^T$: this does not really make sense. In fact, we are trying to approximate a matrix which is highly perturbed with a randomly generated matrix. 

It would arguably make more sense to initialize $H$ at zero, so that, when optimizing over the entries of $H_{:,1}$ at the first step, we only try to approximate the matrix $A$ itself. 
It turns out that this simple strategy allows to obtain a faster initial convergence than the random initialization strategy.  However, we observe the following: this solution tends to have a very particular structure where the first factor is dense and the next ones are sparser. 
The explanation is that the first factor is given more importance since it is optimized first hence it will be close to the best rank-one approximation of $A$, which is in general positive 
(if $A$ is irreducible, by Perron-Frobenius and Eckart-Young theorems). 
Hence initializing $H$ at zero tends to produce unbalanced factors. 
However, this might be desirable in some cases as the next factors are in general significantly sparser than with random initialization. 
To illustrate this, let us perform the following numerical experiment: we use the CBCL face data set (see section~\ref{sec5}) that contains 2429 facial images, 19 by 19 pixels each. Let us construct the nonnegative matrix $X \in \mathbb{R}^{361 \times 2429}$ where each column is a vectorized image. 
Then, we construct the matrix $A = XX^T \in \mathbb{R}^{361 \times 361}$ that contains the similarities between the pixel intensities among the facial images. 
Hence symNMF of $A$ will provide us with a matrix $H$ where each column of $H$ corresponds to a `cluster' of pixels sharing some similarities. Figure~\ref{CBCLzerorand} shows the columns of $H$ obtained (after reshaping them as images) with zero initialization (left) and random initialization (right) with $r=49$ as in \cite{LS99}. 
We observe that the solutions are very different, although the relative approximation error $||A-HH^T||_F/||A||_F$ are similar 
(6.2\% for zero initialization vs. 7.5\% for random initialization, after 2000 iterations). 
Depending on the application at hand, one of the two solutions might be more desirable: for example, for the CBCL data set, it seems that the solution obtained with zero initialization is more easily interpretable as facial features, while with the random initialization it can be interpreted as average/mean faces.   
\begin{figure*}
	\begin{center}
	\begin{tabular}{cc}
	\includegraphics[width=6cm]{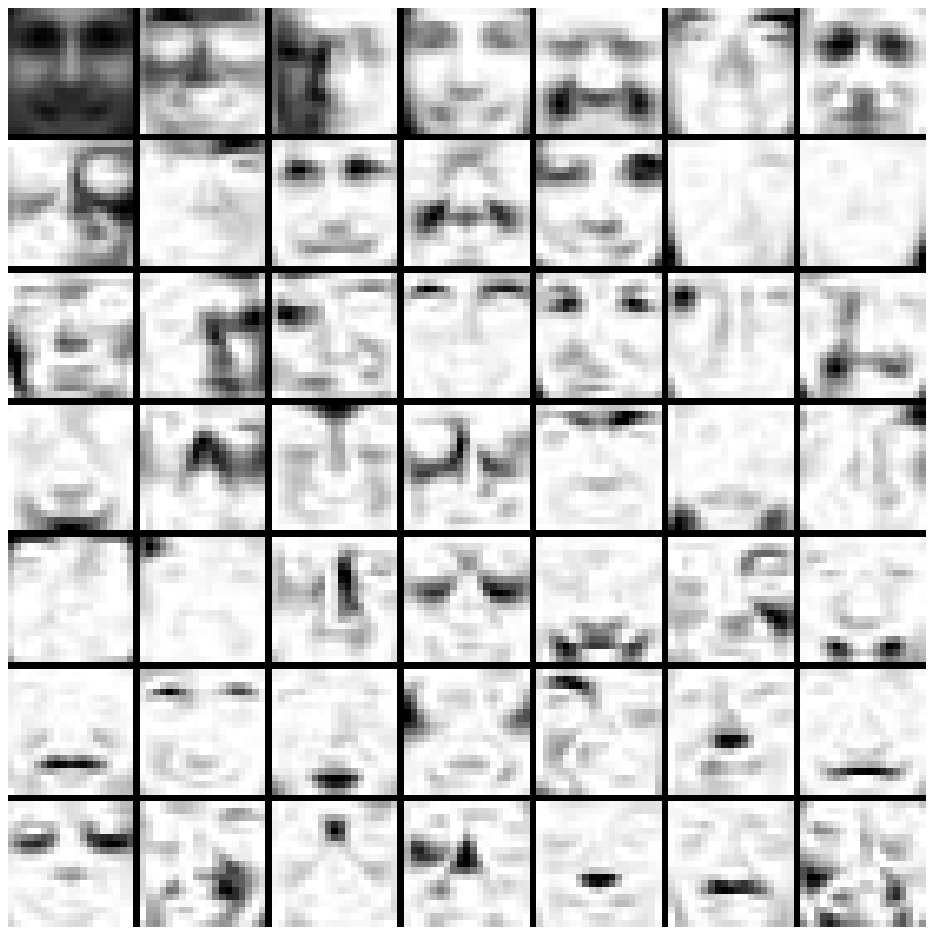} 
	 &  
	\includegraphics[width=6cm]{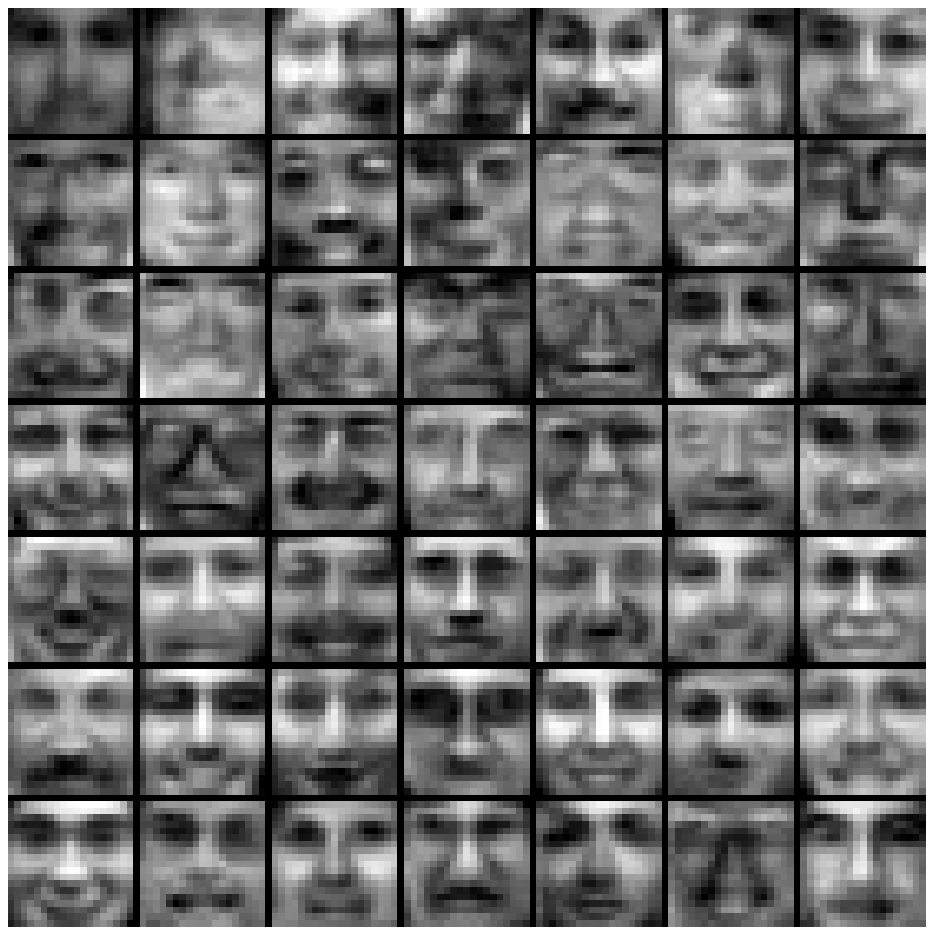} \\
	\end{tabular} 
\end{center}
\caption{Comparison of the basis elements obtained with symNMF on the CBCL data set ($r = 49$) 
with (left) zero initialization and (right) random initialization.} 
\label{CBCLzerorand} 
\end{figure*}

This example also illustrates the sensitivity of Algorithm~\ref{alg:cyclicCD} to initialization: different initializations can lead to very different solutions. This is an unavoidable feature for any algorithm trying to find a good solution to an NP-hard problem at a relatively low computational cost.  

Finally, we would like to point out that the ability to initialize our algorithm at zero is a very nice feature. 
In fact, since $H = 0$ is a (first-order) stationary point of~\eqref{symNMF}, \emph{this shows that our coordinate descent method can escape some first-order stationary points, because it uses higher-order information}. 
For example, any gradient-based method cannot be initialized at zero (the gradient is 0), 
also the ANLS-based algorithm from~\cite{kuang2013symnmf} cannot escape from zero.

\paragraph{Convergence} 

By construction, the objective function is nonincreasing under the updates of Algorithm~\ref{alg:cyclicCD} while it is bounded from below. 
Moreover, since our initial estimate $H_0$ is initially scaled~\eqref{scale}, we have $||A - H_0H_0^T||_F \leq ||A||_F$ and therefore any iterate $H$ of Algorithm~\ref{alg:cyclicCD} satisfies 
\[
||HH^T||_F - ||A||_F \leq ||A - HH^T||_F \leq ||A - H_0H_0^T||_F \leq ||A||_F . 
\]
Since $H \geq 0$, we have for all $k$ 
\[
 ||H_{:k}H_{:k}^T||_F
\leq 
||\sum_{k=1}^r H_{:k}H_{:k}^T||_F
= 
||HH^T||_F, 
\] 
which implies that $||H_{:k}||_2 \leq \sqrt{2||A||_F}$ for all $k$ hence all iterates of Algorithm~\ref{alg:cyclicCD} belong in a compact set. 
Therefore, Algorithm~\ref{alg:cyclicCD} generates a converging subsequence (Bolzano-Weierstrass theorem).  
(Note that, even if the initial iterate is not scaled, all iterates belong to a compact set, replacing $2||A||_F$ by $||A||_F + ||A - H_0H_0^T||_F$.) 

Unfortunately, in its current form, it is difficult to prove convergence of our algorithm to a stationary point. 
In fact, to guarantee the convergence of an exact cyclic coordinate method to a stationary point, three sufficient conditions are 
(i)~the objective function is continuously differentiable over the feasible set, 
(ii)~the sets over which the blocks of variables are updated are compact as well as convex\footnote{An alternative assumption to the condition (ii) under which the same result holds is when the function is monotonically nonincreasing in the interval from one iterate to the next~\cite{B99b}.}, and 
(iii)~the minimum computed at each iteration for a given block of variables is uniquely attained ; see Prop.~2.7.1 in~\cite{B99,B99b}. 
Conditions (i-ii) are met for Algorithm~\ref{alg:cyclicCD}. Unfortunately, it is not necessarily the case that the minimizer of a fourth order polynomial is unique. (Note however that for a randomly generated polynomial, this happens with probability 0. We have observed numerically that this in fact never happens in our numerical experiments, although there are counter examples.) 

A possible way to obtain convergence is to apply the maximum block improvement (MBI) method, that is, at each iteration, only update the variable that leads to the largest decrease of the objective function~\cite{CHLZ12}. Although this is theoretically appealing, this makes the algorithm computationally much more expensive hence much slower in practice. 
(A possible fix is to use MBI not for every iteration, but every $T$th iteration for some fixed $T$.) 

In all our numerical experiments, we have always observed that the sequence of iterates generated by Algorithm~\ref{alg:cyclicCD} converged to a unique limit point. In that case, we can prove that this limit point is a stationary point. 
\begin{theorem} \label{convCDsymNMF} 
Let $(H_{(0)}, H_{(1)}, \dots)$ be a sequence of iterates generated by Algorithm~\ref{alg:cyclicCD}. 
If that sequence converges to a unique accumulation point, it is a stationary point of symNMF~\eqref{symNMF}. 
\end{theorem}
\begin{proof}
This proof follows similar arguments as the proof of convergence of exact cyclic CD for NMF~\cite{hsieh2011fast}.  
 Let $\bar{H}$ be the accumulation point of the sequence $(H_{(0)}, H_{(1)}, \dots)$, that is, 
\[
\lim_{k \rightarrow \infty} H_{(k)} \quad = \quad \bar{H}. 
\]
Note that, by construction, 
\[
F( H_{(1)} ) \geq F( H_{(2)} ) \geq \dots \geq F( \bar{H} ). 
\] 
Note also that we consider that only one variable has been updated between $H_{(k+1)}$ and $H_{(k)}$.

Assume $\bar{H}$ is not a stationary point of~\eqref{symNMF}: therefore, there exists $(i,j)$ such that 
\begin{itemize}
\item $\bar{H}_{i,j} = 0$ and $\nabla F(\bar{H})_{i,j} < 0$, or 
\item $\bar{H}_{i,j} > 0$ and $\nabla F(\bar{H})_{i,j} \neq 0$. 
\end{itemize}
In both cases, since $F$ is smooth, there exists $p \neq 0$ such that 
\[
F( \bar{H} + p E^{ij} ) =  F(  \bar{H} ) - \epsilon < F(  \bar{H} ) , 
\]
for some $\epsilon > 0$, where $E^{ij}$ is the matrix of all zeros except at the $(i,j)$th entry where it is equal to one and $\bar{H} + p E^{ij} \geq 0$.  

Let us define $(H_{(n_0)},H_{(n_1)}, \dots)$ a subsequence of $(H_{(0)}, H_{(1)}, \dots)$ as follows: 
  $H_{(n_k)}$ is the iterate for which the $(i,j)$th entry is updated to obtain $H_{(n_k+1)}$. Since Algorithm~\ref{alg:cyclicCD} updates  the entries of $H$ column by column, we have $n_k = (j-1)n+i-1+nrk$ for $k = 0,1,\dots$. 

By continuity of $F$ and the convergence of the  sequence $H_{(n_k)}$, 
there exists $K$ sufficiently large so that for all $k > K$: 
\begin{equation} \label{pr1}
F( H_{(n_k)} + p E^{ij} ) 
\; \leq \;   
F(  \bar{H} ) - \frac{\epsilon}{2} . 
\end{equation} 
In fact, the continuity of $F$ implies that for all $\xi > 0$, there exists $\delta > 0$ sufficiently small 
such that $||\bar{H} - H_{(n_k)}||_2 < \delta \Rightarrow |F(\bar{H}) - F(H_{(n_k)})| < \xi$. It suffices to choose $n_k$ sufficiently large so that $\delta$ is sufficiently small (since $H_{(n_k)}$ converges to $\bar{H}$) for the value $\xi = \epsilon/2$. 

Let us flip the sign of~\eqref{pr1} and add $F( H_{(n_k)})$ on both sides to obtain
\[
 F( H_{(n_k)} ) - F( H_{(n_k)} + p E^{ij} ) 
 \; \geq \;
 F( H_{(n_k)} ) - F(  \bar{H} ) + \frac{\epsilon}{2} . 
\]
By construction of the subsequence, the $(i,j)$th entry of $H_{(n_k)}$ is updated first 
(the other entries are updated afterward) to obtain $H_{(n_{k+1})}$ which implies that 
\[
F( H_{(n_{k+1})} ) \leq F( H_{(n_{k}+1)} ) \leq F( H_{(n_k)} + p E^{ij} ) 
\] 
hence 
\begin{eqnarray*}
F( H_{(n_k)} ) - F( H_{(n_{k+1})} ) 
& \geq &
 F( H_{(n_k)} ) - F( H_{(n_k)} + p E^{ij} ) 
 \\ 
 & \geq & F( H_{(n_k)} ) - F(  \bar{H} ) + \frac{\epsilon}{2} \\
 & \geq & \frac{\epsilon}{2}, 
\end{eqnarray*} 
since $F(  \bar{H} ) \leq F( H_{(n_k)} )$. We therefore have that for all $k > K$, 
\[
F( H_{(n_{k+1})} ) \leq F( H_{(n_k)} ) - \frac{\epsilon}{2}, 
\]
a contradiction since $F$ is bounded below. 
\end{proof} 

Note that Theorem~\ref{convCDsymNMF} is useful in practice since it can easily be checked whether Algorithm~\ref{alg:cyclicCD} converges to a unique accumulation point, plotting for example the norm between the different iterates. 



\section{Numerical results}
\label{sec5}
This section shows the effectiveness of Algorithm~\ref{alg:cyclicCD} on several data sets compared to the state-of-the-art techniques. It is organized as follows. 
In section~\ref{ds}, we describe the real data sets and, in section~\ref{testedalgo}, the tested symNMF algorithms.   
In section~\ref{set}, we describe the settings we use to compare the symNMF algorithms.   
In section~\ref{compstate}, we provide and discuss the experimental results.


\subsection{Data sets} \label{ds}

We will use exactly the same data sets as in~\cite{GG12}. Because of space limitation, we only give the results for one value of the factorization rank $r$, more numerical experiments are available on the arXiv version of this paper~\cite{VGLZD15}. 
In~\cite{GG12}, authors use four dense data sets and six sparse data sets to compare several NMF algorithms. 
In this section, we use these data sets to generate similarity matrices $A$ on which we compare the different symNMF algorithms.  
Given a nonnegative data set $X \in \mathbb{R}^{m \times n}_+$, we construct the symmetric similarity matrix $A = X^TX \in \mathbb{R}^{n \times n}_+$, so that the entries of $A$ are equal to the inner products between data points. 
Table~\ref{dip} summarizes the dense data sets, corresponding to widely used facial images in the data mining community. Table~\ref{dtm} summarizes the characteristics of the different sparse data sets, corresponding to document datasets and described in details in~\cite{ZG05}.  

\begin{center}
\begin{table}[h!]
\begin{center}
\caption{Image datasets. }
\label{dip}  
\begin{tabular}{|c|c|c|c|}
\hline 
Data &            $\#$ pixels &  $m$  &  $n$    \\ \hline \hline
ORL$^1$    &   $112 \times 92$  & 10304 & 400    \\ 
Umist$^2$  &   $112 \times 92$  & 10304 & 575    \\ 
CBCL$^3$ &  $19 \times 19$  & 361  & 2429     \\ 
Frey$^2$  &   $28 \times 20$ & 560 & 1965   \\ 
\hline
\end{tabular} \\
\begin{flushleft}
\footnotesize
$^1$ \url{http://www.cl.cam.ac.uk/research/dtg/attarchive/facedatabase.html}\\
$^2$ \url{http://www.cs.toronto.edu/~roweis/data.html}\\
$^3$ \url{http://cbcl.mit.edu/cbcl/software-datasets/FaceData2.html} 
\end{flushleft}
\end{center}
\end{table}
\end{center}

\begin{center}
\begin{table}[h!] 
\begin{center}
\caption{Text mining data sets  (sparsity is given as the percentage of zeros).} 
\label{dtm}
\begin{tabular}{|c|c|c|c|c|c|c|} 
\hline
Data &   $m$  &  $n$ &    $\#\text{nonzero}$ & sparsity $X$  & sparsity $X^TX$ \\ \hline \hline
classic &  7094   & 41681  & 223839  & 99.92 &  99.50 \\ 
sports &    8580 & 14870  & 1091723 & 99.14 & 84.51 \\ 
reviews &   4069  & 18483  & 758635  & 98.99 & 84.24 \\ 
hitech &    2301 & 10080  & 331373 & 98.57& 80.32 \\ 
ohscal &   11162  & 11465  & 674365 & 99.47 & 91.58  \\ 
la1 &    3204 & 31472  & 484024  & 99.52 & 95.72 \\ 
\hline
\end{tabular}
\end{center}
\end{table}
\end{center}

\subsection{Tested symNMF algorithms} \label{testedalgo}

We compare the following algorithms  
\begin{enumerate} 
\item (Newton) This is the Newton-like method from~\cite{kuang2013symnmf}. 

\item (ANLS) This is the method based on the ANLS method for NMF adding the penalty $||W-H||_F^2$ in the objective function (see Introduction) from~\cite{kuang2013symnmf}. Note that ANLS has the drawback to depend on a parameter that is nontrivial to tune, namely, the penalty parameter for the term $||W-H||_F^2$ in the objective function (we used the default tuning strategy recommended by the authors).

\item (tSVD) This method, recently introduced in~\cite{HSS14}, first computes the rank-$r$ truncated SVD of $A \approx A_r = U_r \Sigma_r U_r^T$ where $U_r$ contains the first $r$ singular vectors of $A$ and $\Sigma_r$ is the $r$-by-$r$ diagonal matrix containing the first $r$ singular values of $A$ on its diagonal. 
Then, instead of solving~\eqref{symNMF}, the authors solve a `closeby' optimization problem replacing $A$ with $A_r$
\[
\min_{H \geq 0} || A_r - HH^T ||_F . 
\] 
Once the truncated SVD is computed, each iteration of this method is extremely cheap as the main computational cost is in a matrix-matrix product $B_r Q$, where $B_r = U_r \Sigma_r^{1/2}$ and $Q$ is an $r$-by-$r$ rotation matrix, which can be computed in $O(n r^2)$ operations. 
Note also that they use the initialization $H_0 = \max(0, B_r)$ --we flipped the signs of the columns of $U_r$ to maximize the $\ell_2$ norm of the nonnegative part~\cite{BAK08}.  

\item (BetaSNMF) This algorithm is presented in~\cite[Algorithm~4]{he2011symmetric}, and is based on multiplicative updates (similarly as for the original NMF algorithm proposed by Lee and Seung~\cite{LS2}). 
Note that we have also implemented the multiplicative update rules from~\cite{yang2012quadratic} (and already derived in~\cite{long2007relational}). 
However, we do not report the numerical results here because it was outperformed by BetaSNMF in all our numerical experiments, an observation already made in \cite{he2011symmetric}. 


\item (CD-X-Y) This is Algorithm~\ref{alg:cyclicCD}. 
X is either `Cyclic' or `Shuffle' and indicates whether the columns of $H$ are optimized in a cyclic way or if they are shuffled randomly before each iteration. 
Y is for the initialization: Y is `rand' for random initialization and is `0' for zero initialization; see section~\ref{init} for more details. Hence, we will compare four variants of Algorithm~\ref{alg:cyclicCD}: 
CD-Cyclic-0, CD-Shuffle-0, CD-Cyclic-Rand and CD-Shuffle-Rand.  

Because Algorithm~\ref{alg:cyclicCD} requires to perform many loops ($nr$ at each step), 
Matlab is not a well-suited language. 
Therefore, we have developed a C implementation, that can be called from Matlab (using Mex files). 
Note that the algorithms above are better suited for Matlab since the main computational cost resides in matrix-matrix products, and in solving linear systems of equations (for ANLS and Newton).  

\end{enumerate} 

Newton and ANLS are both available from \url{http://math.ucla.edu/~dakuang/}, while we have implemented tSVD and BetaSNMF ourselves. 

For all algorithms using random initializations for the matrix $H$, we used the same initial matrices. Note however that, in all the figures presented in this section, we will display the error after the first iteration, which is the reason why the curves do not start at the same value.

\subsection{Experimental setup} \label{set}

In order to compare for the average performance of the different algorithms, 
we denote $e_{\min}$ the smallest error obtained by all algorithms over all initializations, and define
\begin{equation} \label{measerr}
E(t) = \frac{e(t) - e_{\min}}{||A||_F - e_{\min}} ,   
\end{equation}
where $e(t)$ is the error $||A-HH^T||_F$ achieved by an algorithm for a given initialization within $t$ seconds (and hence $e(0) = ||A-H_0 H_0^T||_F$ where $H_0$ is the initialization). 
The quantity $E(t)$ is therefore a normalized measure of the evolution of the objective function of a given algorithm on a given data set. 

The advantage of this measure is that it separates better the different algorithms, when using a log scale, since it goes to zero for the best algorithm (except for algorithms that are initialized randomly as we will report the average value of $E(t)$ over several random initializations; see below). 
We would like to stress out that the measure $E(t)$ from~\eqref{measerr} has to be interpreted with care. 
In fact, an algorithm for which $E(t)$ converges to zero simply means that it is the algorithm able to find the best solution among all algorithms (in other words, to identify a region with a better local minima). In fact, the different algorithms are initialized with different initial points: in particular, tSVD uses an SVD-based initialization. 
It does not necessarily mean that it converges the fastest: to compare (initial) convergence, one has to look at the values $E(t)$ for $t$ small. 
However, the measure $E(t)$ allows to better visualize the different algorithms. For example, displaying the relative error $||A-HH^T||/||A||_F$ allows to compare the initial convergence, but then the errors for all algorithms tend to converge at similar values and it is difficult to identify visually which one converges to the best solution. 

For the algorithms using random initialization (namely, Newton, ANLS, CD-Cyclic-Rand and CD-Shuffle-Rand), we will run the algorithms 10 times and report the average value of $E(t)$. 
For all data sets, we will run each algorithm for 60 seconds. 

All tests are performed using Matlab on a PC Intel CORE i5-4570 CPU @3.2GHz $\times$ 4, with 7.7G RAM. The codes are available online from \url{https://sites.google.com/site/nicolasgillis/}. 


\begin{remark}[Computation of the error]
Note that to compute $||A-HH^T||_F$, one should not compute $HH^T$ explicitly (especially if $A$ is sparse) and use instead 
\begin{eqnarray*}
||A-HH^T||_F^2  & = & ||A||_F^2 - 2 \langle A,  H H^T \rangle + ||H H^T||_F^2 \\
 & = & ||A||_F^2 - 2 \langle A H,  H \rangle  + ||H^T H||_F^2 .
\end{eqnarray*} 
\end{remark}

\subsection{Comparison} \label{compstate}

We now compare the different symNMF algorithms listed in section~\ref{testedalgo} according to the measure given in~\eqref{measerr} on the data sets described in section~\ref{testedalgo}, and on synthetic data sets.

\subsubsection{Real data sets} 

We start with the real data sets. 

\paragraph{Dense image data sets}  

 Figure~\ref{realdense60} displays the results for the dense real data sets. 
Table~\ref{denseriter} gives the number of iterations performed by each algorithm within the 60 seconds, and Table~\ref{densererr}  the final relative error $||A-HH^T||/||A||_F$ in percent. 

\begin{table*}[ht] 
\begin{center}
\caption{Average number of iterations performed by each algorithm within 60 seconds for the dense real data sets.} 
\label{denseriter} 
\small 
\begin{tabular}{|c||c|c|c|c|c|c|c|c|}
\hline
$r = 60$ & ANLS & Newton & tSVD & BetaSNMF & CD-Cyc.-0 & CD-Shuf.-0 & CD-Cyc.-Rand & CD-Shuf.-Rand \\ 
\hline 
ORL &  6599   &  2861   & 32685 & 33561 & 2514 & 2501 &  2261   &  2306   \\
Umist &  5305   &  1431   & 25982 & 18218 & 1500 & 1496 &  1406   &  1357   \\
CBCL &   473   &     4   & 12109 & 1281 & 116 & 113 &    95   &    95   \\
Frey &   705   &     5   & 15373 & 1876 & 166 & 157 &   138   &   139   \\
\hline
\end{tabular}
\end{center}
\end{table*}

\begin{table*}[ht]  
\begin{center}
\caption{Average relative error in percent ($100 \, * \, ||A-HH^T||_F/||A||_F$) of the final solution obtained by each algorithm within 60 seconds for the dense real data sets. For algorithms based on random initializations, the standard deviation is given.}  
\footnotesize   \begin{tabular}{|c||c|c|c|c|c|c|c|c|}
\hline
$r = 60$ & ANLS & Newton & tSVD & BetaSNMF & CD-Cyc.-0 & CD-Shuf.-0 & CD-Cyc.-Rand & CD-Shuf.-Rand \\ 
\hline
ORL & 0.288  $\pm$ 4e-3   & 0.341     & \textbf{0.141} & 0.143 $\pm$ 3e-4 & 0.142    & 0.143    & 0.165  $\pm$ 6e-4   & \textbf{0.141} $\pm$ 8e-5 \\
Umist & 0.718  $\pm$ 0.023   & 0.365     & \textbf{0.041} & 0.073 $\pm$ 7e-4 & 0.043    & 0.044    & 0.108  $\pm$ 2e-3   & 0.042   $\pm$ 2e-4  \\
CBCL & 0.254  $\pm$ 2e-3   & 4.52     & \textbf{0.046} & 0.679 $\pm$ 3e-3 & 0.169    & 0.176    & 0.751  $\pm$ 7e-3   & 0.157     $\pm$ 1e-3 \\
Frey & 0.083   $\pm$ 6e-4  & 4.88      & \textbf{0.057} & 0.510 $\pm$ 2e-3 & 0.105    & 0.107    & 0.765  $\pm$ 4e-3   & 0.124   $\pm$ 2e-3  \\
\hline
\end{tabular}
\label{densererr}
\end{center}
\end{table*}

\normalsize

We observe the following: 
\begin{itemize}

\item In all cases, tSVD performs best and is able to generate the solution with the smallest objective function value among all algorithms. This might be a bit surprising since it works only with an approximation of the original data: it appears that for these real dense data sets, this approximation can be computed efficiently and allows tSVD to converge extremely fast to a very good solution. 

One of the reasons tSVD is so effective is because each iteration is $n$ times cheaper (once the truncated SVD is computed) hence it can perform many more iterations; see Table~\ref{denseriter}. 
Another crucial reason is that image data sets can be very well approximated by low-rank matrices (see section~\ref{synthd} for a confirmation of this behavior). Therefore, for images, tSVD is the best method to use as it provides a very good solution extremely fast.

\item When it comes to initial convergence, CD-Cyclic-0 and CD-Shuffle-0 perform best: they are able to generate very fast a good solution. In all cases, they are the fastest to generate a solution at a relative error of $1\%$ of the final solution of tSVD. Moreover, the fact that tSVD does not generate any solution  as long as  the truncated SVD is not computed could be critical for larger data sets. For example, for CBCL with $n = 2429$ and $r = 60$, the truncated SVD takes about 6 seconds to compute while, in the mean time, CD-Cyclic-0 and CD-Shuffle-0 generate a solution with relative error of 
$0.3\%$  from the final solution obtained by tSVD after 60 seconds. 

\item For these data sets, CD-Cyclic-0 and CD-Shuffle-0 perform exactly the same: for the zero initialization, it seems that shuffling the columns of $H$ does not play a crucial role. 

\item When initialized randomly, we observe that the CD method performs significantly better with random shuffling. 
Moreover, CD-Shuffle-Rand converges initially slower than CD-Shuffle-0 but is often able to converge to a better solution; in particular for the ORL and Umistim data sets.

\item Newton converges slowly, the main reason being that each iteration is very costly, namely $O(n^3 r)$ operations.

\item ANLS performs relatively well: it never converges initially faster than CD-based approaches but is able to generate a better final solution for the Frey data set.   

\item BetaSNMF does not perform well on these data sets compared to tSVD and CD methods, although performing better than ANLS and 2 out of 4 times better than ANLS.  

\item For algorithms based on random initializations, the standard deviation between several runs is rather small, illustrating the fact that these algorithms converge to solutions with similar final errors. 

\end{itemize} 

\textbf{Conclusion}: for image data sets, tSVD performs the best. However, CD-Cyclic-0 allows a very fast initial convergence and can be used to obtain very quickly a good solution. 


\begin{figure*}
	\begin{center}
	\begin{tabular}{cc}
 \includegraphics[width=8cm]{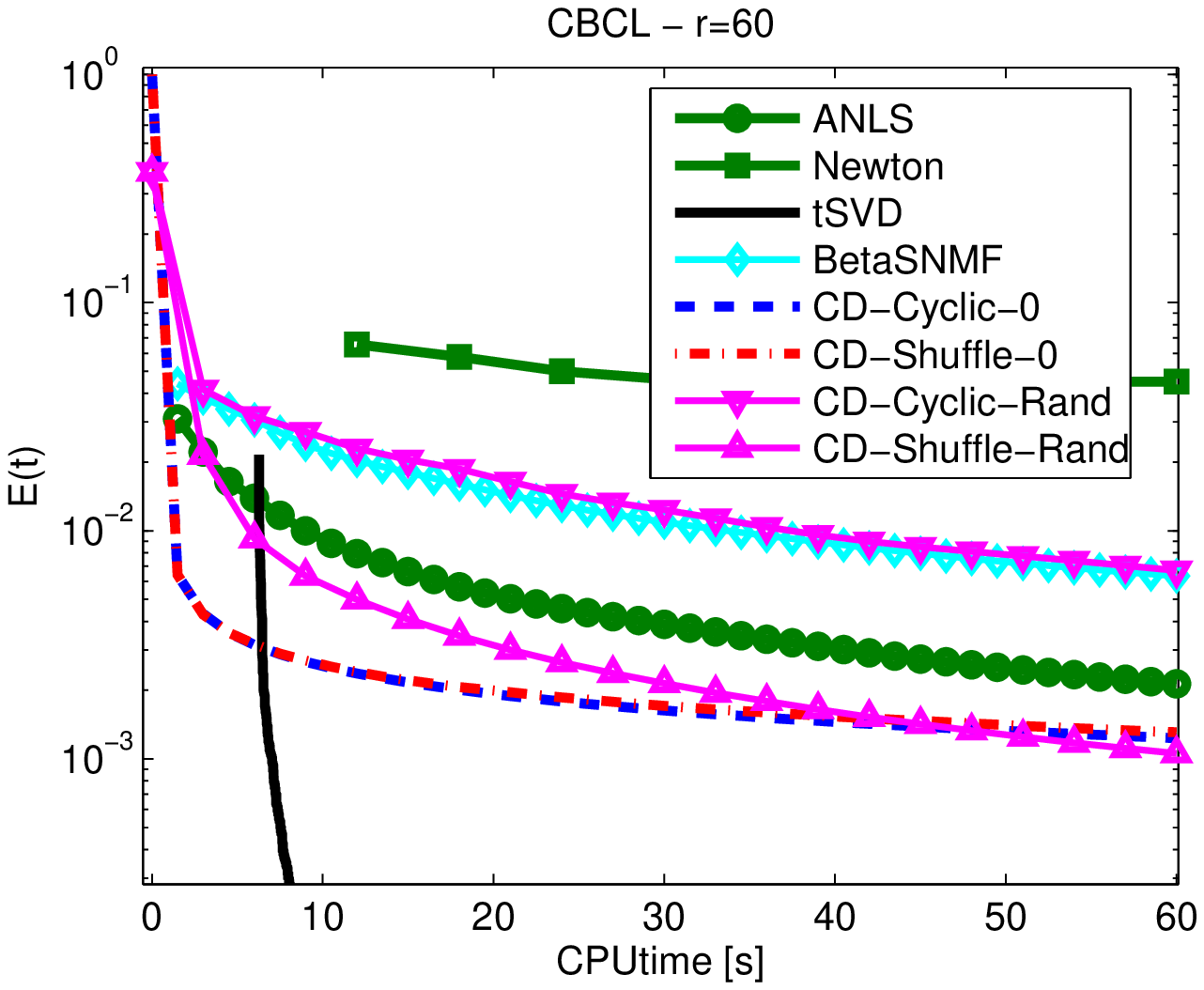} & \includegraphics[width=8cm]{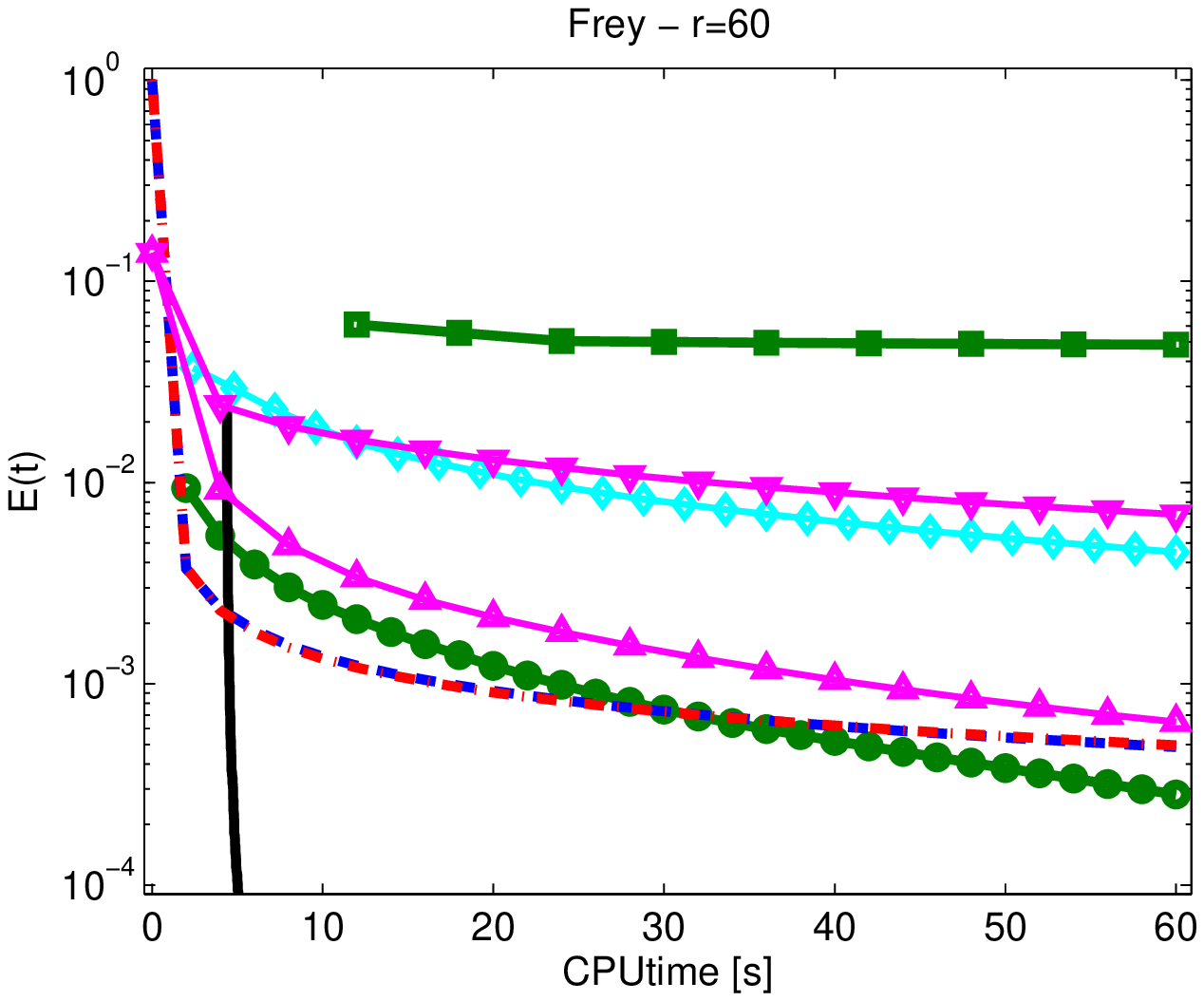} \\
	  \includegraphics[width=8cm]{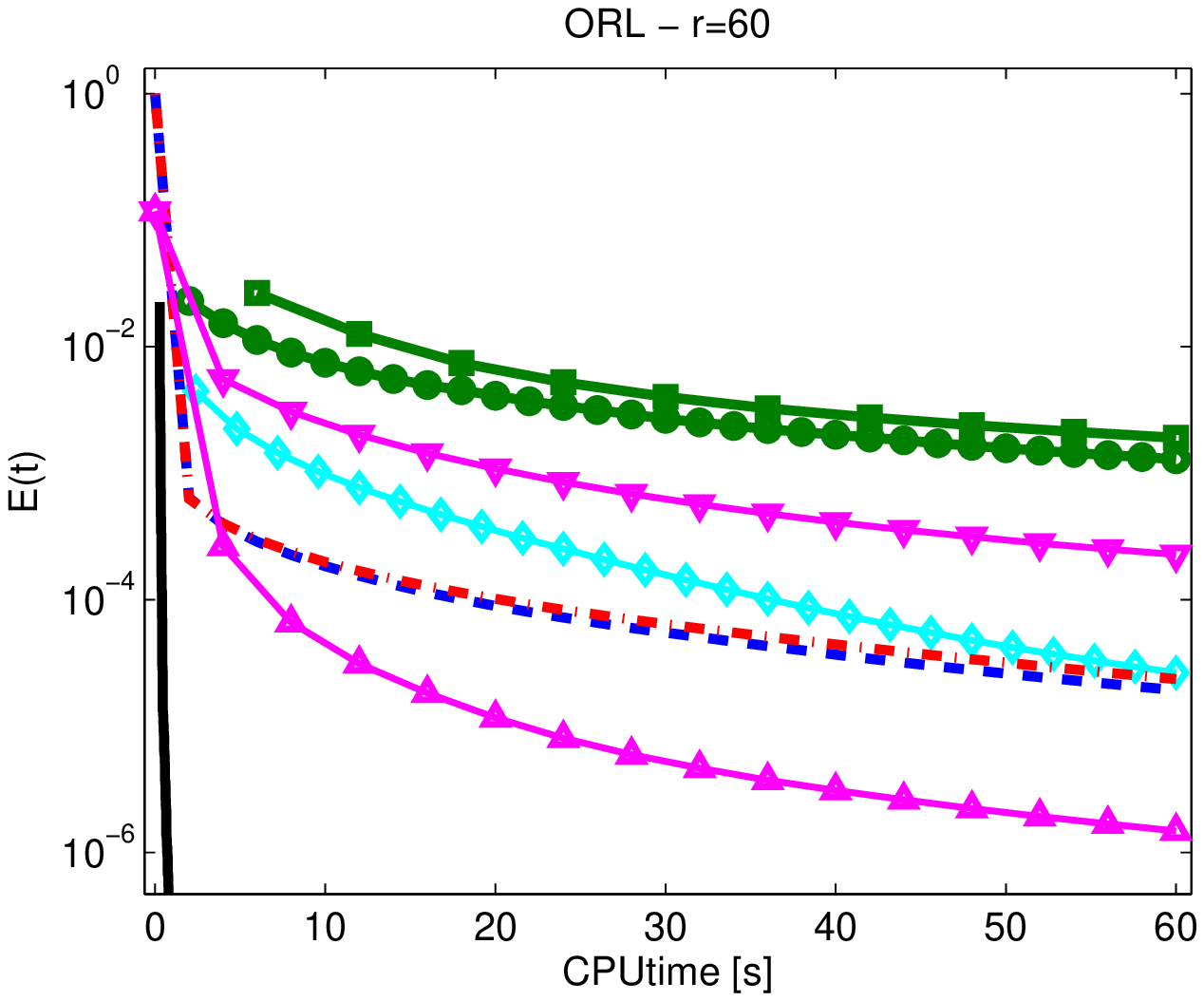} & \includegraphics[width=8cm]{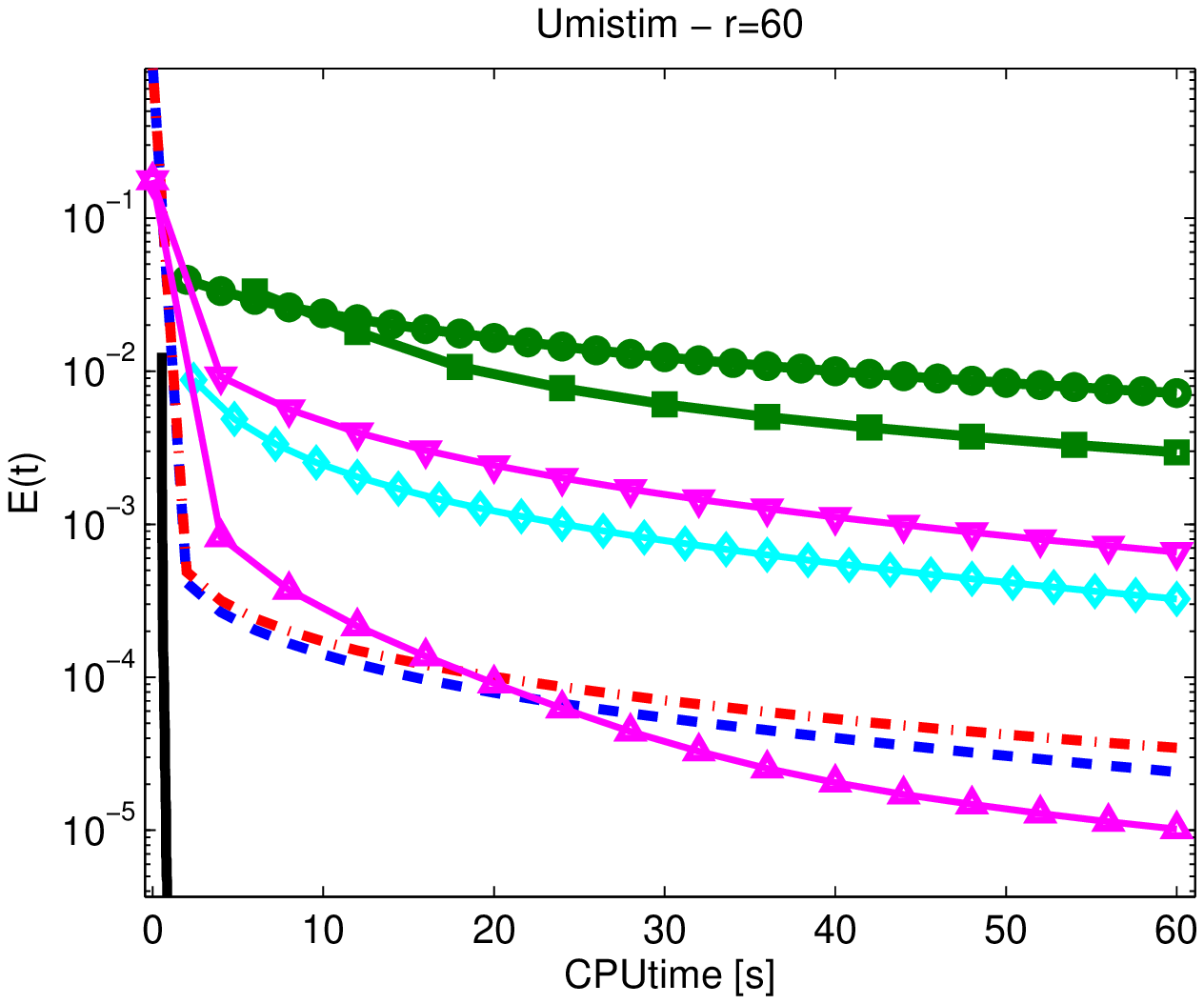} 
	\end{tabular} 
\end{center}
\caption{Evolution of the measure~\eqref{measerr} of the different symNMF algorithms on the dense real data sets for $r = 60$.} 
\label{realdense60} 
\end{figure*}

\paragraph{Sparse document data sets}   

 Figure~\ref{realsparse30} displays the results for the real sparse data sets. Table~\ref{sparseriter} gives the number of iterations performed by each algorithm within the 60 seconds, 
and Table~\ref{sparsererr}  the final relative error $||A-HH^T||/||A||_F$ in percent. 

For some data sets (namely, la1 and reviews), computing the truncated SVD of $A$ was not possible with Matlab within 60 seconds hence tSVD was not able to return any solution; see Remark~\ref{rem3} for a discussion. Moreover,  Newton is not displayed because it is not designed for sparse matrices and runs out of memory~\cite{kuang2013symnmf}. 

\begin{table*}[ht] 
\begin{center}
\caption{Average number of iterations performed by each algorithm within 60 seconds for the sparse real data sets.} 
\label{sparseriter} 
\begin{tabular}{|c||c|c|c|c|c|c|c|}
\hline
$r = 30$ & ANLS  & tSVD & BetaSNMF & CD-Cyc.-0 & CD-Shuf.-0 & CD-Cyc.-Rand & CD-Shuf.-Rand \\
\hline
classic &    41.3 & 1212 & 237 & 44 & 44 &    44  &    44  \\
sports &    34  & 4330 & 66 & 23 & 23 &    23  &    23  \\
reviews &    16.7 & 0 & 41 & 13 & 13 &    13  &    13  \\
hitech &    61  & 8334 & 115 & 37 & 37 &    37  &    37  \\
ohscal &    91.7 & 7855 & 199 & 61 & 61 &    61  &    61  \\
la1 &    16  & 0 & 43 & 15 & 15 &    15  &    15  \\
\hline
\end{tabular}
\end{center} 
\end{table*}

\begin{table*}[ht]  
\begin{center}
\caption{Average relative error in percent ($100 \, * \, ||A-HH^T||_F/||A||_F$) of the final solution obtained by each algorithm within 60 seconds for the sparse real data sets. For algorithms based on random initializations, the standard deviation is given.} 
\label{sparsererr}
\small 
\begin{tabular}{|c||c|c|c|c|c|c|c|c|}
\hline
$r = 30$ & ANLS  & tSVD & BetaSNMF & CD-Cyc.-0 & CD-Shuf.-0 & CD-Cyc.-Rand & CD-Shuf.-Rand \\ 
\hline 
classic & 99.99 $\pm$ 1e-4  & 39.8  & 38.1 $\pm$ 0.14 & \textbf{37.6}  & 37.8  & \textbf{37.6} $\pm$ 0.09  & 37.7 $\pm$ 0.09  \\
sports & 99.9\hspace{0.15cm} $\pm$ 1e-3 & 19.2 & 20.1 $\pm$ 0.28 & \textbf{17.5}  & 17.7  & \textbf{17.5} $\pm$ 0.11  & 17.7 $\pm$ 0.10  \\
reviews & 99.9\hspace{0.15cm} $\pm$ 7e-4 & / & 20.0 $\pm$ 0.56 &\textbf{16.3}  & 16.4  & \textbf{16.3} $\pm$ 0.10  & \textbf{16.3} $\pm$ 0.08  \\
hitech & 99.5\hspace{0.15cm} $\pm$ 4e-3 & 33.3 & 31.3 $\pm$ 0.22 & 30.5  & 30.5  & \textbf{30.4} $\pm$ 0.09  & \textbf{30.4} $\pm$ 0.08  \\
ohscal & 99.95 $\pm$ 1e-3  & 22.2 & 21.6 $\pm$ 0.11 & \textbf{20.9}  & 21.0  & \textbf{20.9} $\pm$ 0.04 & \textbf{20.9} $\pm$ 0.04  \\
la1 & 99.9\hspace{0.15cm} $\pm$ 8e-4 & / & 34.9 $\pm$ 0.32 & \textbf{31.9}  & 32.0  & 32.1 $\pm$ 0.10  & 32.0  $\pm$ 0.10 \\
\hline
\end{tabular}
\end{center}
\end{table*}

\begin{figure*}
	\begin{center}
	\begin{tabular}{cc}
	\includegraphics[width=8cm]{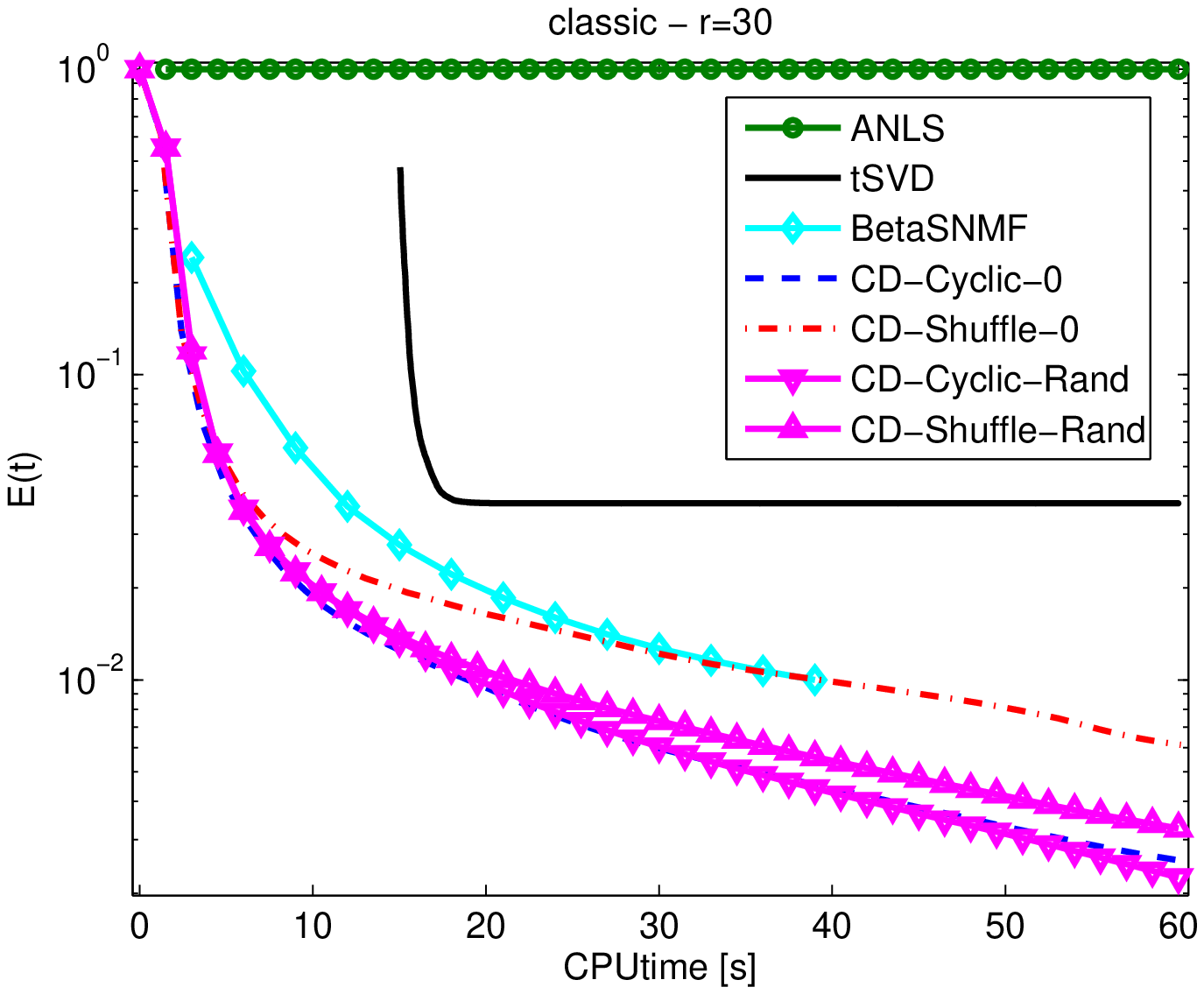} & \includegraphics[width=8cm]{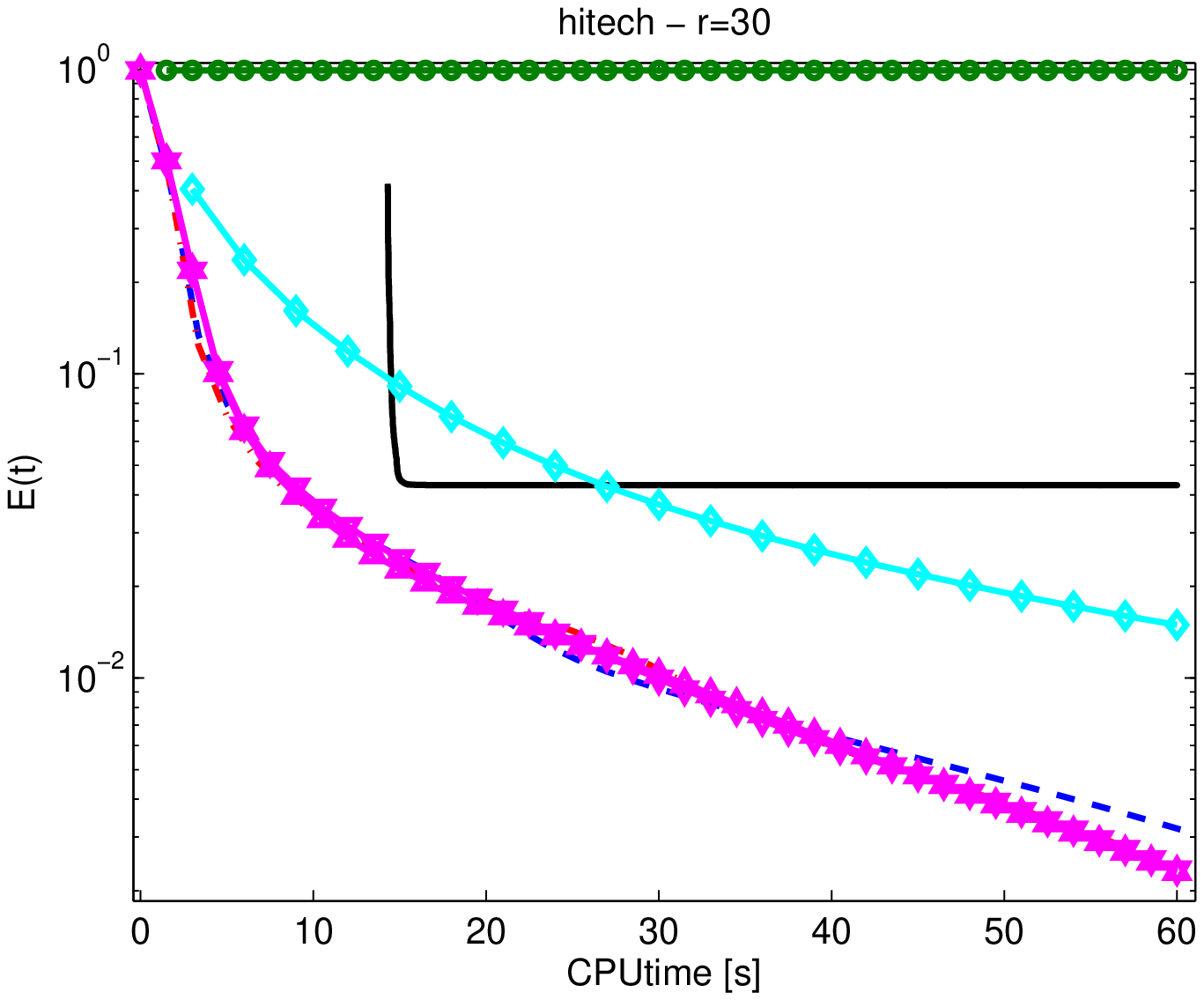} \\
	\includegraphics[width=8cm]{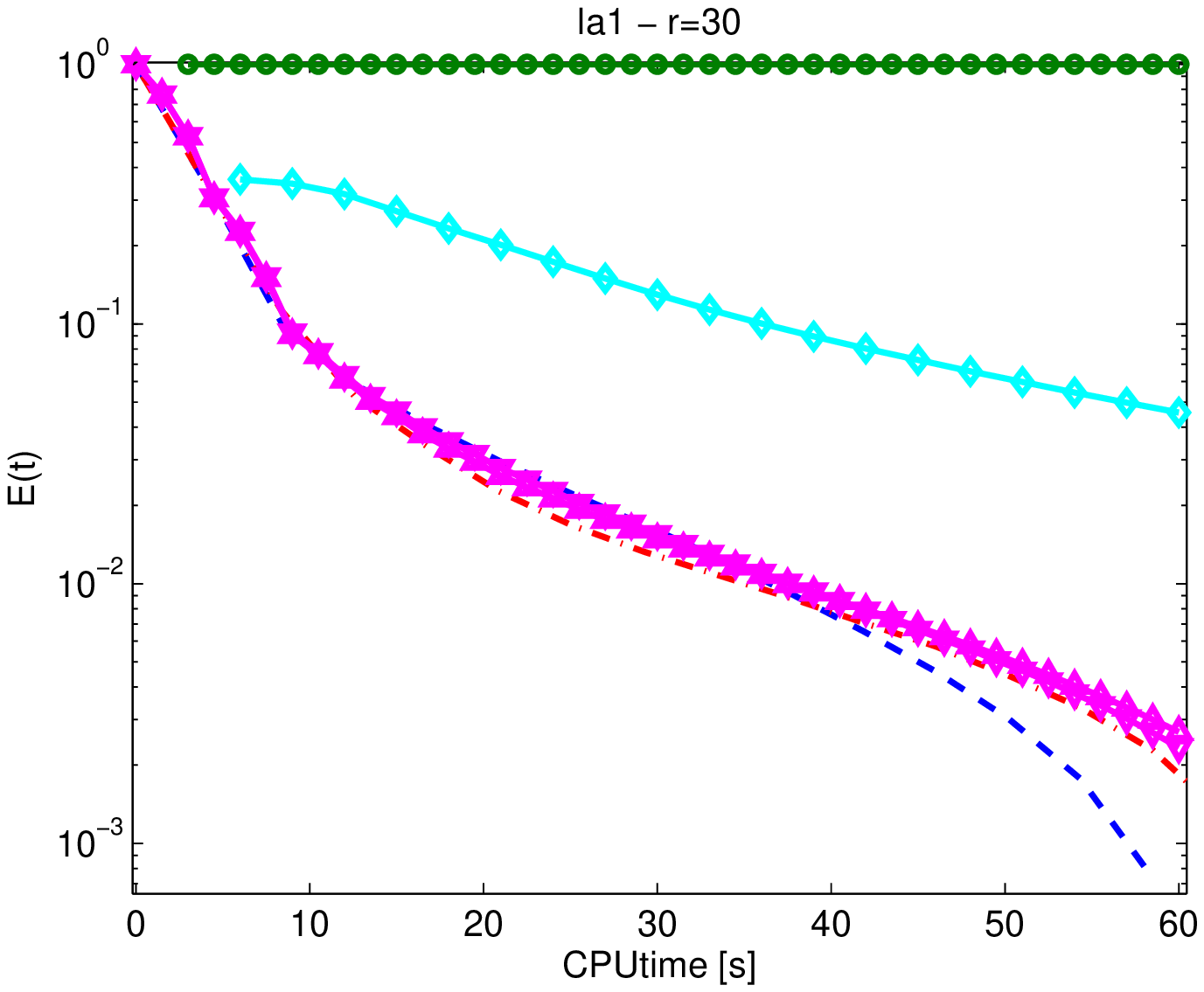} &	\includegraphics[width=8cm]{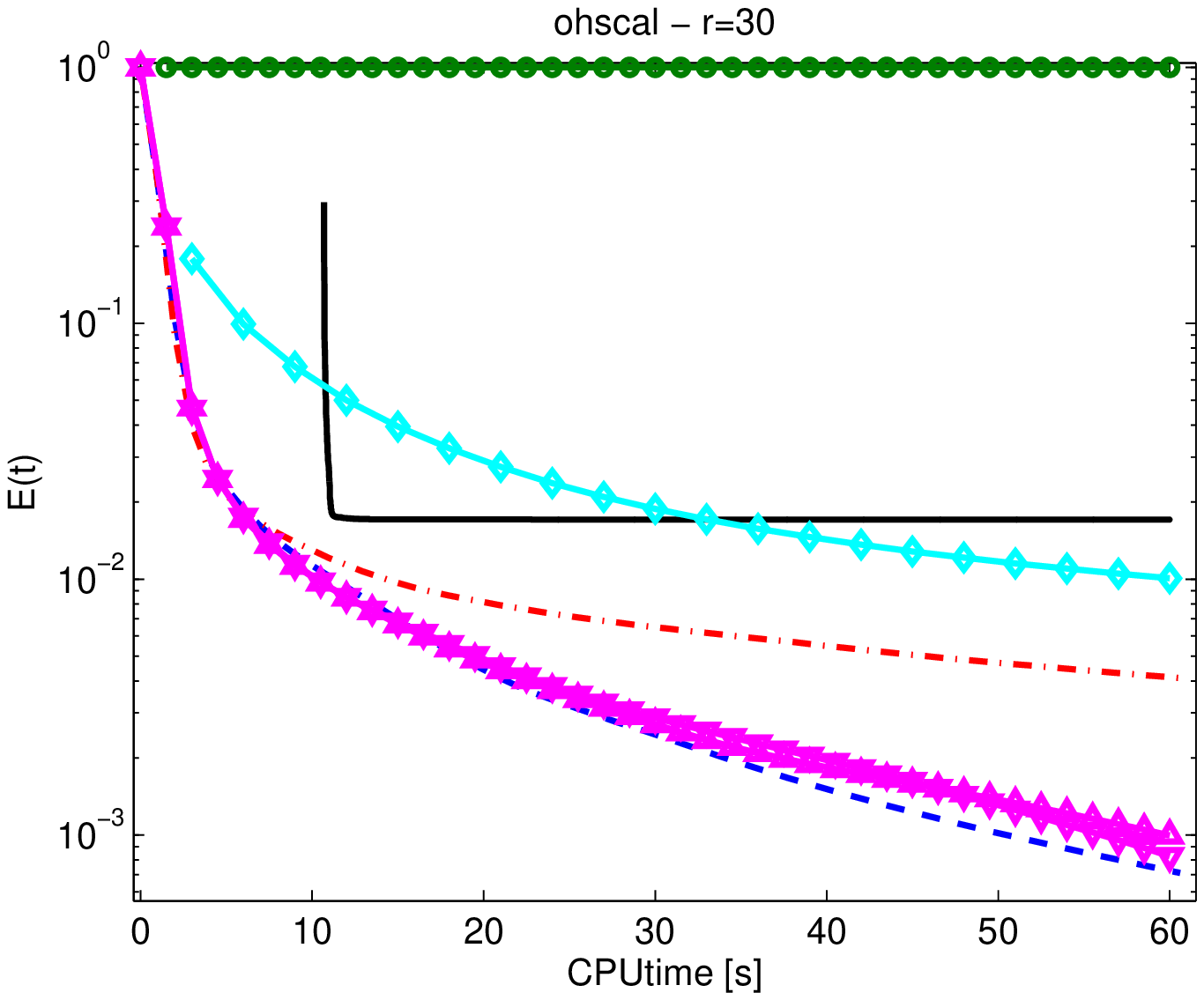}  \\
	  \includegraphics[width=8cm]{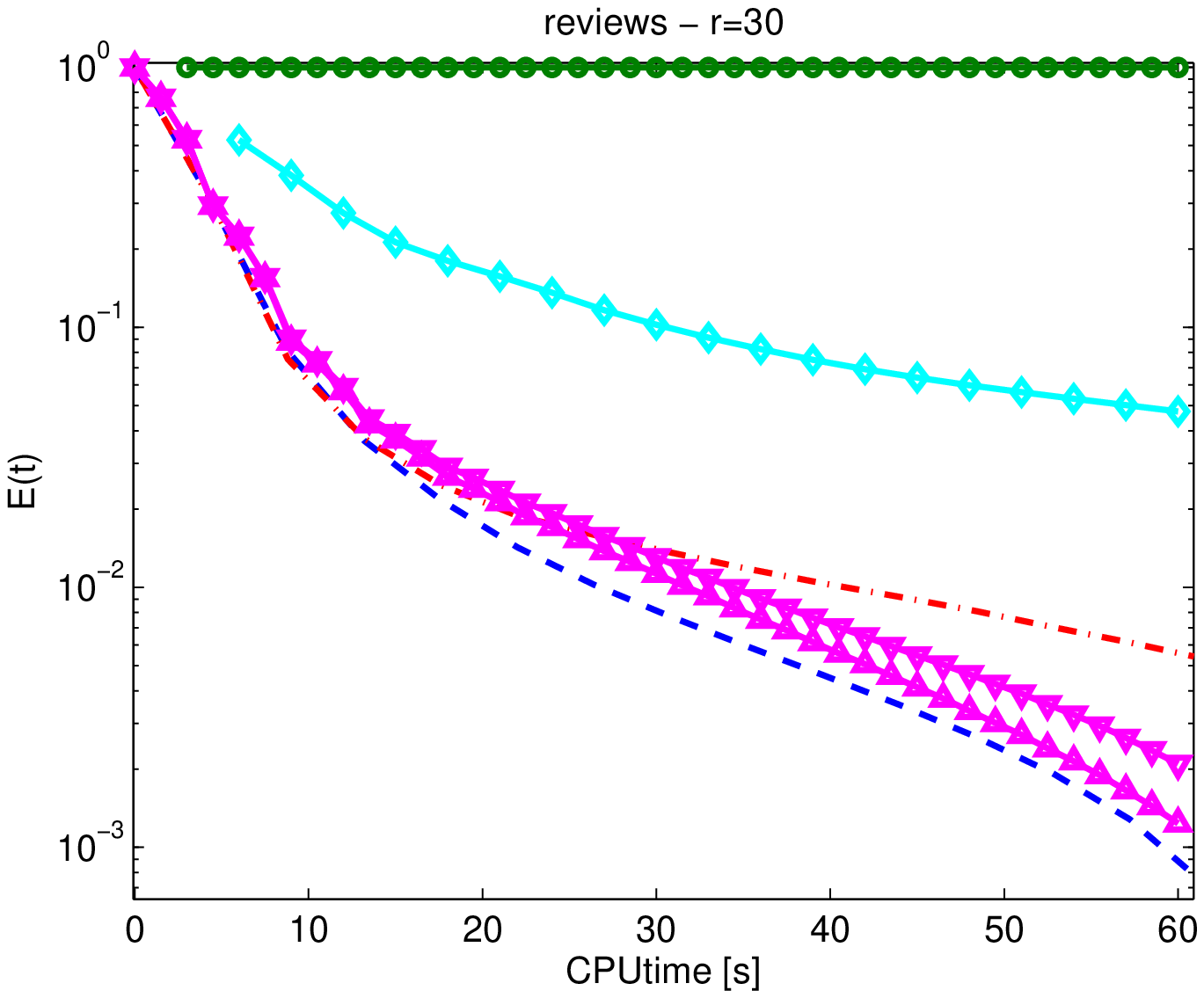} & \includegraphics[width=8cm]{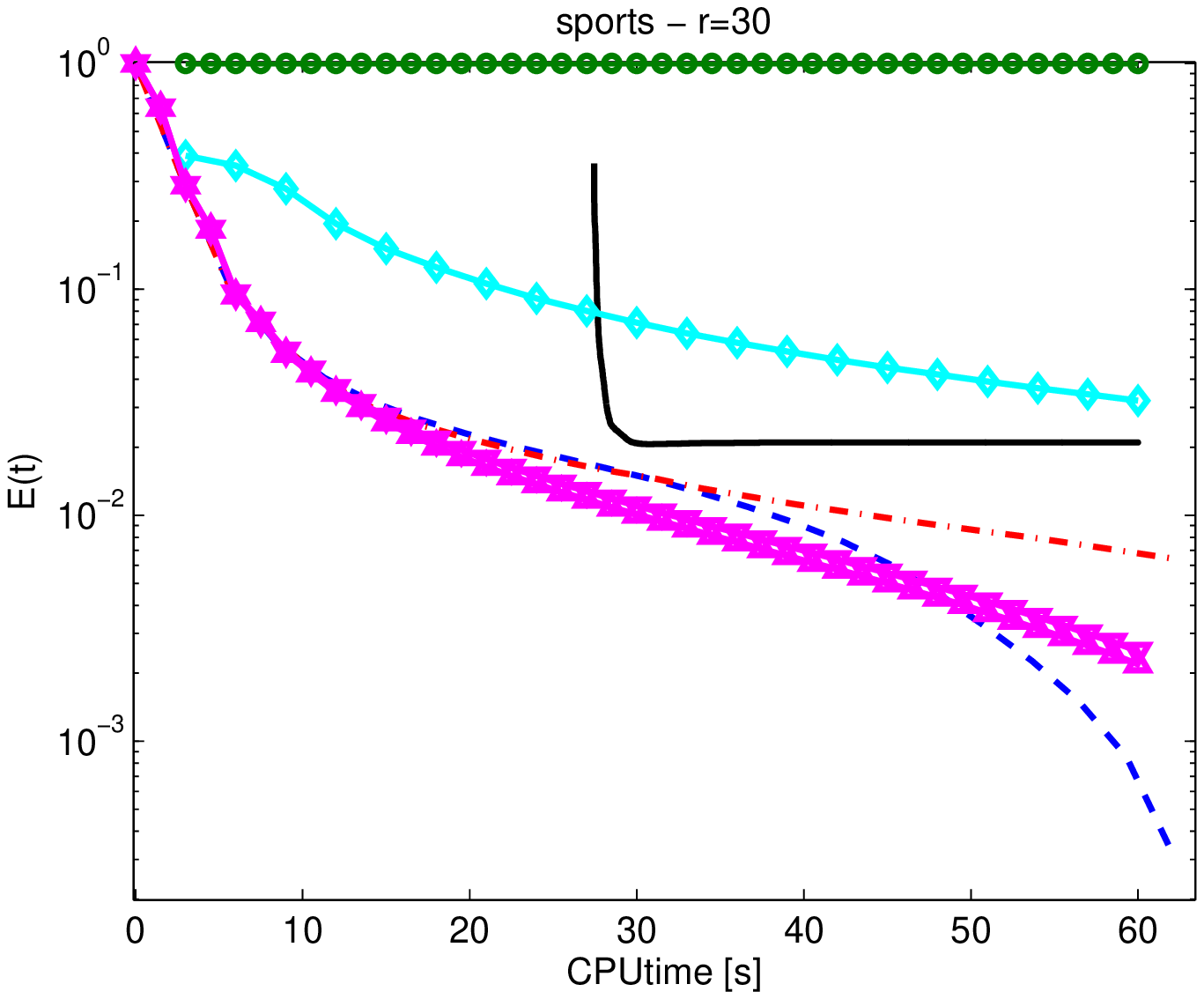} 
	\end{tabular} 
\end{center}
\caption{Evolution of the measure~\eqref{measerr} of the different symNMF algorithms on real sparse data sets for $r = 30$.} 
\label{realsparse30} 
\end{figure*}

We observe the following: 
\begin{itemize}

\item tSVD performs very poorly. The reason is twofold: (1)~the truncated SVD is very expensive to compute and (2)~sparse matrices are usually not close to being low-rank hence tSVD converges to a very poor solution (see section~\ref{synthd} for a confirmation of this behavior). 

\item ANLS performs very poorly and is not able to generate a good solution. In fact, it has difficulties to decrease the objective function (on the figures, it seems it does not decrease, but it actually decreases very slowly).  

\item BetaSNMF performs better than ANLS but does not compete with CD methods. (Note that, for the classic data sets, BetaSNMF was stopped prematurely because there was a division by zero which could have been avoided but we have strictly used the description of Algorithm~4 in~\cite{he2011symmetric}). 

\item All CD-based approaches are very effective and perform similarly. 
It seems that, in these cases, nor the initialization nor the order in which the columns of $H$ are updated plays a significant role. 

However, we observe that in all cases, $E(t)$ converges to a value between $10^{-2}$ and $10^{-4}$, never to a smaller value. This means that the best solution is always obtained with random initialization 
(in fact, for algorithms initialized randomly, Figure~\ref{realsparse30} reports the average over 10 runs) 
but, on average, random initialization performs similarly as the initialization with zero. 

\end{itemize}

\textbf{Conclusion}: for sparse document data sets, CD-based approaches outperform significantly the other tested methods.

\begin{remark}[SVD computation in tSVD] \label{rem3} 
It has to be noted that, in our numerical experiments, the matrix $A$ is constructed using the formula $A = X^TX$, where the columns of the matrix $X$ are the data points.  
In other words, we use the simple similarity measure $y^T z$ between two data points $y$ and $z$.  
In that case, the SVD of $A$ can be obtained from the SVD of $X$, hence can be made 
(i) more efficient (when $X$ has more columns than rows, that is, $m \ll n$), and 
(ii) numerically more accurate (because the condition number of $X^TX$ is equal to the square of that of $X$); see, e.g.,~\cite[Lecture~31]{trefethen1997numerical}. 
Moreover, in case of sparse data, this avoids the fill-in, as  observed in Table~\ref{dtm} where $X^TX$ is denser than $X$. 
Therefore, in this particular situation when $A=X^TX$ and $X$ is sparse and/or $m \ll n$, it is much better to compute the SVD of $A$ based on the SVD of $X$. Table~\ref{svdsMat} gives the computational time in both cases. 
\begin{table}[ht]  
\begin{center}
\caption{Computational time required to compute the rank-30 truncated SVD of $X$ and $X^TX$ using Matlab.} 
\label{svdsMat}
\begin{tabular}{|c||c|c|c|c|c|c|}
\hline
   \hspace{-0.2cm} \texttt{svds(.,30)} \hspace{-0.2cm} & classic & hitech  & la1  & ohscal & reviews   & sports \\ 
\hline 
\texttt{X'*X}  &  17.14   & 18.54   &   63.33  &    15    &     67.32   &   31.77    \\
\texttt{X}  & 5.55   &    0.82   &    3.08  &    2.87 &       1.39   &     2.98 \\
\hline
\end{tabular}
\end{center}
\end{table} 
In this particular scenario, it would make sense to use tSVD as an initialization procedure for CD methods to obtain rapidly a good initial iterate. However, looking at Figure~\ref{realsparse30} and Table~\ref{sparsererr} indicates that this would not necessarily be advantageous for the CD-based methods in all cases. For example, for the classic data set, tSVD would achieve a relative error of 39.8\% within about 6 seconds while CD methods obtain a similar relative error within that computing time. 
For the hitech data set however, this would be rather helpful since tSVD would only take about 1 second to obtain a relative error of 33.3\% while CD methods require about 9 seconds to do so.   

However, the goal of this paper is to provide an efficient algorithm for the general symNMF problem, without assuming any particular structure on the matrix $A$ (in practice the similarity measure between data points is usually not simply their inner product). . Therefore, we have not assumed that the matrix $A$ had this particular structure and only provide numerical comparison in that case. 
\end{remark}

\begin{remark}[Low-rank models for full-rank matrices] 
Although sparse data sets are usually not low rank, it still makes sense to try to find a low-rank structure that is close to a given data set, as this often allows to extract some pertinent information.  
In particular, in document classification and clustering, low-rank models have proven to be extremely useful; see the discussion in the Introduction and the references therein.  
Another important application where low-rank models have proven extremely useful although 
the data sets are usually not low-rank is \emph{recommender systems}~\cite{KBV09}. 
We also refer the reader to the recent survey~\cite{UHZB14}. 
\end{remark}

\subsubsection{Synthetic data sets: low-rank vs.\@ full rank matrices} \label{synthd}

In this section, we perform some numerical experiments on synthetic data sets. Our main motivation is to confirm the (expected) behavior observed on real data: tSVD performs extremely well for low-rank matrices and poorly on full-rank matrices. 


\paragraph{Low-rank input matrices} 

The most natural way to generate nonnegative symmetric matrices of given cp-rank is to generate $H_*$ randomly and then compute $A = H_*H_*^T$. In this section, we use the Matlab function $H_* = \texttt{rand}(n,r)$ with $n = 500$ and $r = 30, 60$, that is, 
each entry of $H_*$ is generated uniformly at random in the interval [0,1]. 
We have generated 10 such matrices for each rank, and Figure~\ref{denselrsynth} displays the average value for the measure~\eqref{measerr} but we use here $e_{\text{min}} = 0$ since it is the known optimal value. 
\begin{figure*}
	\begin{center}
	\begin{tabular}{cc}
	\includegraphics[width=8cm]{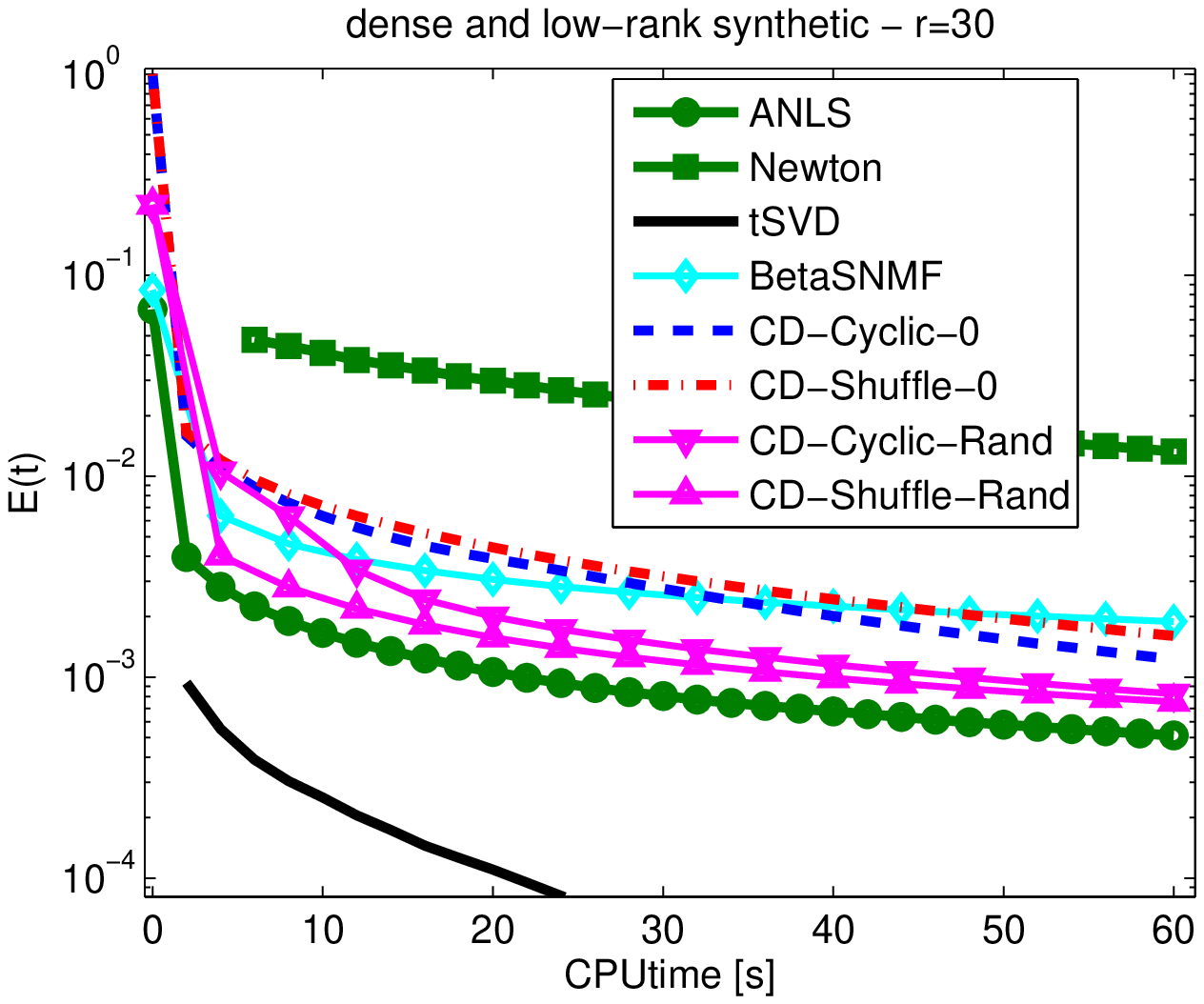} & \includegraphics[width=8cm]{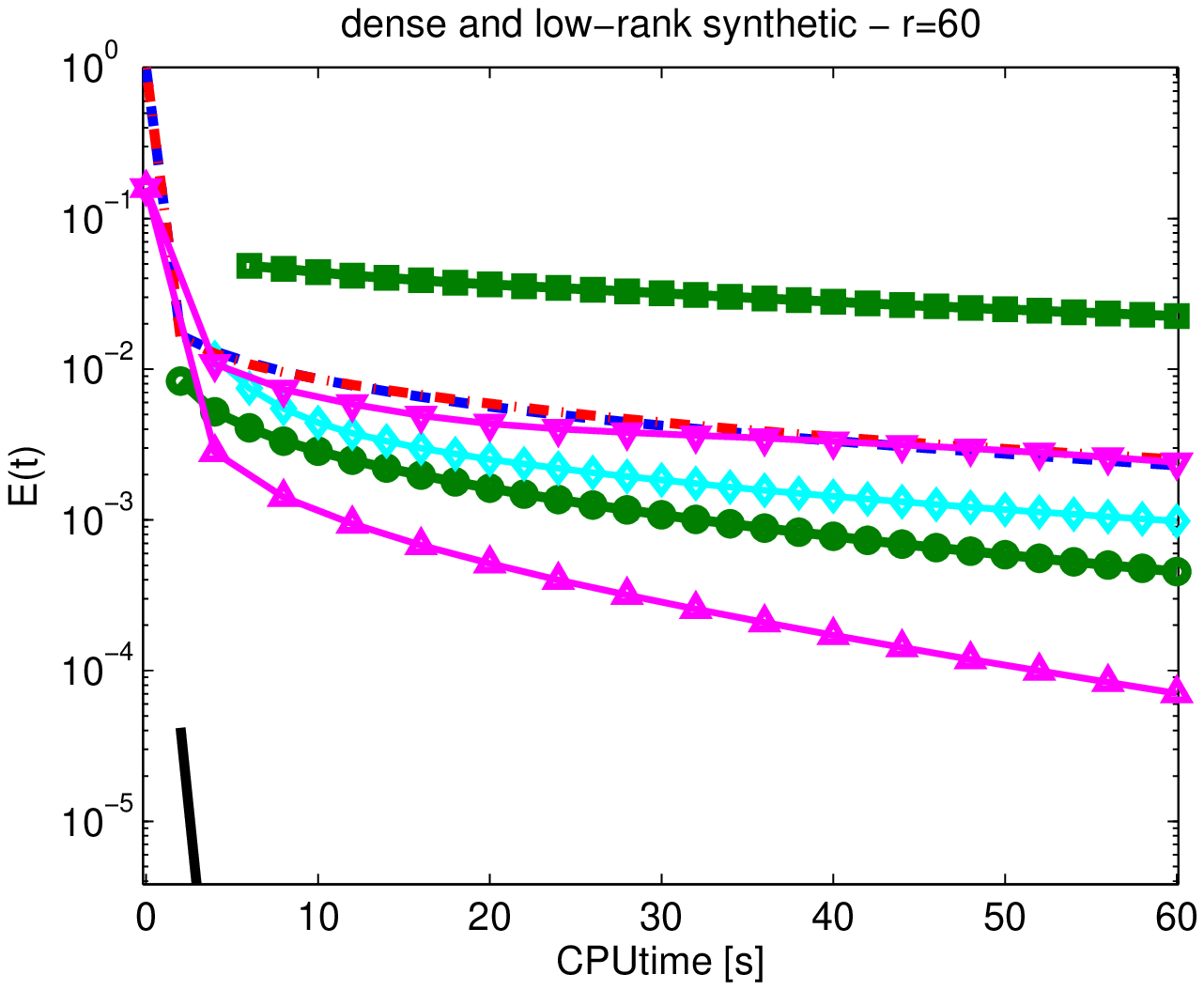} 
	\end{tabular} 
\end{center}
\caption{Evolution of the measure~\eqref{measerr} of the different symNMF algorithms on dense and low-rank synthetic data sets for $r = 30$ (left) and $r = 60$ (right).} 
\label{denselrsynth} 
\end{figure*}


We observe that, in all cases, tSVD outperforms all methods. Moreover, it seems that the SVD-based initialization is very effective. 
The reason is that $A$ has exactly rank $r$ and hence its best rank-$r$ approximation is exact. 
Moreover, tSVD only works in the correct subspace in which $H_*$ belongs  hence converges much faster than the other methods. 

Except for Newton, the other algorithms perform similarly. 
It is worth noting that the same behavior we observed for real dense data sets is present here: 
CD-Shuffle-Rand performs better than CD-Cyclic-Rand, while shuffling the columns of $H$ before each iteration does not play a crucial role with the zero initialization. 


\paragraph{Full-rank input matrices} 

A simple way to generate nonnegative symmetric matrices of full rank is to generate a matrix $B$ randomly and then compute $A = B + B^T$. In this section, we use the Matlab function $B = \texttt{rand}(n)$ with $n = 500$.  
We have generated 10 such matrices for each rank, and Figure~\ref{densefrsynth} displays the average value for the measure $E(t)$ from~\eqref{measerr}. 
 Figure~\ref{densefrsynth} displays the results. 
\begin{figure*}
	\begin{center}
	\begin{tabular}{cc}
	\includegraphics[width=8cm]{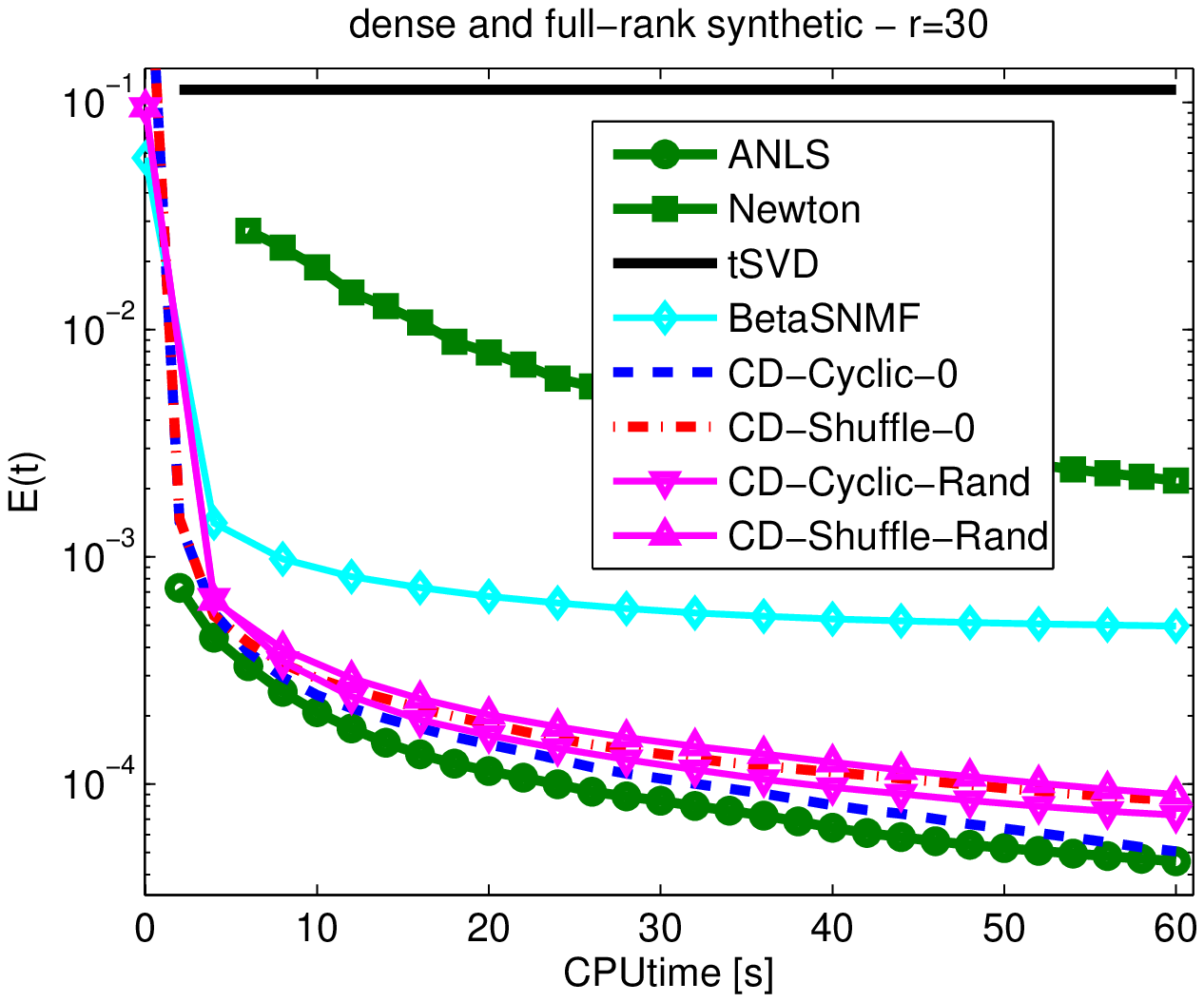} & \includegraphics[width=8cm]{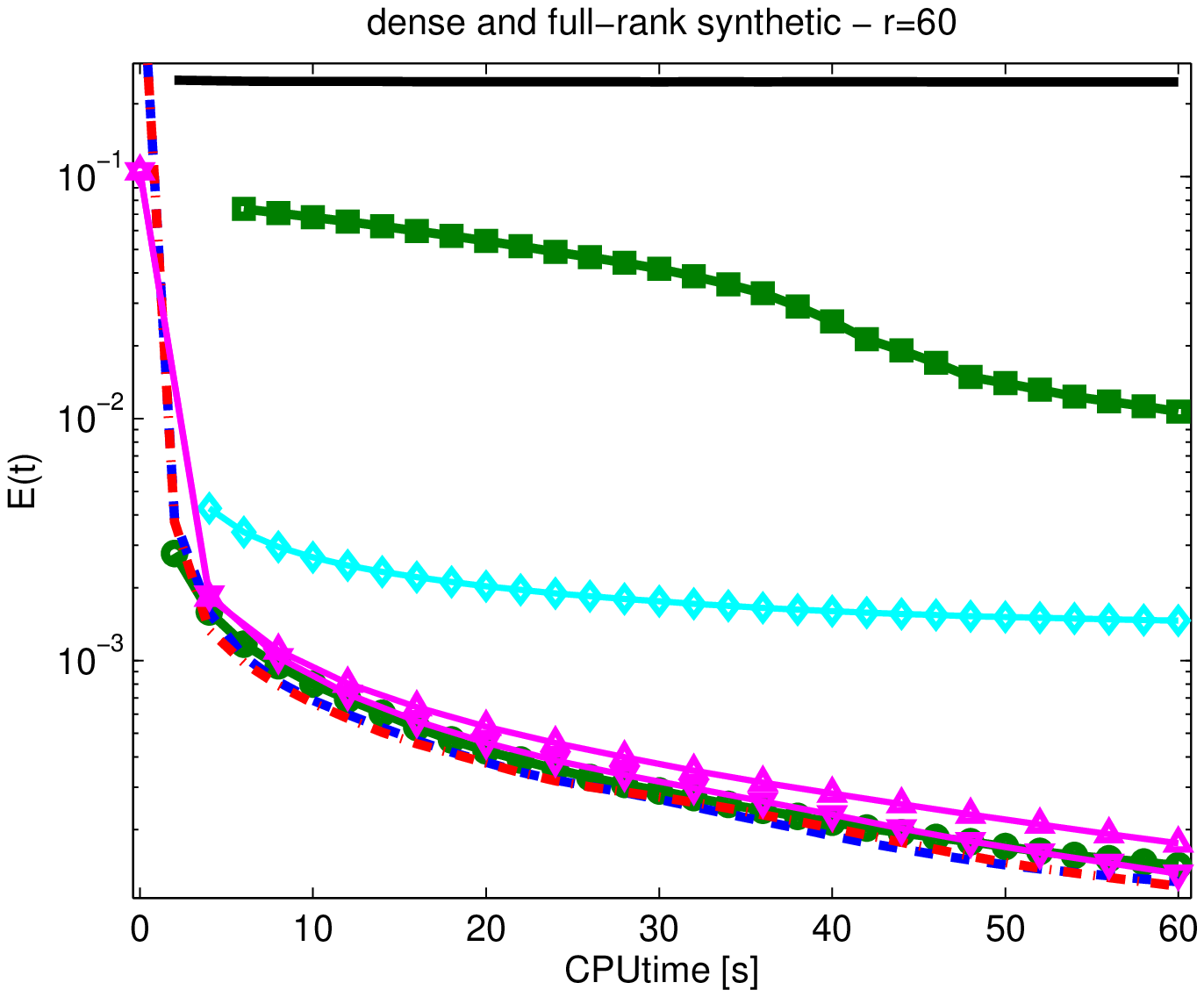} 
	\end{tabular} 
\end{center}
\caption{Evolution of the measure~\eqref{measerr} of the different symNMF algorithms on dense full-rank synthetic data sets for $r = 30$ (left) and $r = 60$ (right).} 
\label{densefrsynth} 
\end{figure*}


We observe that, in all cases, tSVD performs extremely poorly while all other methods (except for Newton and BetaSNMF) perform similarly. 
The reason is that tSVD works only with the best rank-$r$ approximation of $A$, which is poor when $A$ has full rank.

\subsubsection{Summary of results}

Clearly, tSVD and CD-based approaches are the most effective, although ANLS sometimes performs competitively for the dense data sets.  
However, tSVD performs extremely well only when the input matrix is low rank (cf.~low-rank synthetic data sets) or close to being low rank (cf.~image data sets). There are three cases when it performs very poorly: 
\begin{itemize}

\item It cannot perform a symNMF when the factorization rank $r$ is larger than the rank of $A$, that is, when $r > \text{rank}(A)$, which may be necessary for matrices with high cp-rank (in fact, the cp-rank can be much higher than the rank~\cite{BS03}). 

\item If the truncated SVD is a poor approximation of $A$, the algorithm will perform poorly since it does not use any other information; see the results for the full rank synthetic data sets and the sparse real data sets.  

\item The algorithm returns no solution as long as the SVD is not computed. In some cases, the cost of computing the truncated SVD is high and tSVD terminates before any solution to symNMF is produced; see the sparse real data sets. 

\end{itemize}

To conclude, CD-based approaches are overall the most reliable and most effective methods to solve symNMF~\eqref{symNMF}. 
For dense data sets, initialization at zero allows a faster initial convergence, while 
CD-Shuffle-Rand generates in average the best solution and 
CD-Cyclic-Rand does not perform well and is not recommended. 
For sparse data sets, all CD variants perform similarly and outperform the other tested algorithms.

\section{Conclusion and further research}
\label{sec6}

In this paper, we have proposed very efficient exact coordinate descent methods for symNMF~\eqref{symNMF} that performs competitively with state-of-the-art methods.  

Some interesting directions for further research are the following:
\begin{itemize}

\item The study of sparse symNMF, where one is looking for a sparser matrix $H$. A natural model would for example use the sparsity-inducing  $\ell_1$ norm and try to solve 
\begin{equation} \label{ssymNMF}
\min_{H \geq 0} \; \frac{1}{4} ||A - HH^T||_F^2 + \sum_{j=1}^r \Lambda_j ||H_{:,j}||_1 \; , 
\end{equation} 
for some penalty parameter $\Lambda \in \mathbb{R}_+^r$. 
Algorithm~\ref{alg:cyclicCD} can be easily adapted to handle~\eqref{ssymNMF}, by replacing the $b_{ij}$'s with $b_{ij} + \Lambda_j$. In fact, the derivative of the penalty term only influences the constant part in the gradient; see~\eqref{bij}. However, it seems the solutions of \eqref{ssymNMF} are very sensitive to the parameter $\Lambda$ and hence are difficult to tune. 
Note that another way to identify sparser factors is simply to increase the factorization rank $r$, or to sparsify the input matrix $A$ (only keeping the important edges in the graph induced by $A$; see~\cite{BSS13} and the references therein) --in fact, a sparser matrix $A$ induces sparser factors since 
\[
A_{ij} = 0 
 \Rightarrow 
H_{i,:} H_{j,:}^T \approx 0 
 \Rightarrow  
H_{ik} \approx 0 \text{ or } H_{jk} \approx 0 \, \forall k. 
\] 
This is an interesting observation: $A_{ij} = 0$ implies a (soft) orthogonality constraints on the rows of $H$. This is rather natural: if item $i$ does not share any similarity with item $j$ ($A_{ij} = 0$), then they should be assigned to different clusters ($H_{ik} \approx 0 \text{ or } H_{jk} \approx 0 \text{ for all } k$). 

\item The design of more efficient algorithms for symNMF. For example, a promising direction would be to combine the idea from~\cite{HSS14} that use a compressed version of $A$ with very cheap per-iteration cost with our more reliable CD method, 
to combine the best of both worlds. 


\end{itemize}

\bibliographystyle{plain} 
\bibliography{Biography}

\begin{thebibliography}{10}

\bibitem{BSS13}
J.~Batson, D.A. Spielman, N.~Srivastava, and S.H. Teng.
\newblock Spectral sparsification of graphs: theory and algorithms.
\newblock {\em Communications of the ACM}, 56(8):87--94, 2013.

\bibitem{belachew2015solving}
M.T. Belachew and N.~Gillis.
\newblock Solving the maximum clique problem with symmetric rank-one
  nonnegative matrix approximation.
\newblock {\em arXiv:1505.07077}, 2015.

\bibitem{BS03}
A.~Berman and N.~Shaked-Monderer.
\newblock {\em Completely Positive Matrices}.
\newblock World Scientific Publishing, 2003.

\bibitem{B99b}
D.P. Bertsekas.
\newblock {Corrections for the book Nonlinear Programming: Second Edition}.
\newblock \url{http://www.athenasc.com/nlperrata.pdf}, 1999.

\bibitem{B99}
D.P. Bertsekas.
\newblock {\em Nonlinear Programming: Second Edition}.
\newblock Athena Scientific, Massachusetts, 1999.

\bibitem{BAK08}
R.~Bro, E.~Acar, and T.G. Kolda.
\newblock Resolving the sign ambiguity in the singular value decomposition.
\newblock {\em Journal of Chemometrics}, 22(2):135--140, 2008.

\bibitem{Bu09}
S.~Burer.
\newblock {On the copositive representation of binary and continuous nonconvex
  quadratic programs}.
\newblock {\em Math. Prog.}, 120(2):479--495, 2009.

\bibitem{cardano1968ars}
G.~Cardano.
\newblock {\em Ars magna or the rules of algebra}.
\newblock Dover Publications, 1968.

\bibitem{CHLZ12}
B.~Chen, S.~He, Z.~Li, and S.~Zhang.
\newblock Maximum block improvement and polynomial optimization.
\newblock {\em SIAM J. on Optimization}, 22(1):87--107, 2012.

\bibitem{chen2008non}
Y.~Chen, M.~Rege, M.~Dong, and J.~Hua.
\newblock Non-negative matrix factorization for semi-supervised data
  clustering.
\newblock {\em Knowledge and Information Systems}, 17(3):355--379, 2008.

\bibitem{CP09b}
A.~Cichocki and A.-H. Phan.
\newblock {Fast local algorithms for large scale Nonnegative Matrix and Tensor
  Factorizations}.
\newblock {\em IEICE Trans.\@ on Fundamentals of Electronics}, Vol. E92-A
  No.3:708--721, 2009.

\bibitem{DG14}
P.J.C. Dickinson and L.~Gijben.
\newblock On the computational complexity of membership problems for the
  completely positive cone and its dual.
\newblock {\em Computational Optimization and Applications}, 57(2):403--415,
  2014.

\bibitem{NG11}
N.~Gillis.
\newblock {\em Nonnegative Matrix Factorization: Complexity, Algorithms and
  Applications}.
\newblock PhD thesis, Universit\'{e} catholique de Louvain, 2011.
\newblock \url{https://sites.google.com/site/nicolasgillis/}.

\bibitem{GG12}
N.~Gillis and F.~Glineur.
\newblock Accelerated multiplicative updates and hierarchical \textsc{ALS}
  algorithms for nonnegative matrix factorization.
\newblock {\em Neural Computation}, 24(4):1085--1105, 2012.

\bibitem{he2011symmetric}
Zhaoshui He, Shengli Xie, Rafal Zdunek, Guoxu Zhou, and Andrzej Cichocki.
\newblock Symmetric nonnegative matrix factorization: Algorithms and
  applications to probabilistic clustering.
\newblock {\em Neural Networks, IEEE Transactions on}, 22(12):2117--2131, 2011.

\bibitem{Ho2008}
N.-D. Ho.
\newblock {\em Nonnegative Matrix Factorization: Algorithms and Applications}.
\newblock PhD thesis, Universit\'{e} catholique de Louvain, 2008.

\bibitem{hsieh2011fast}
C.-J. Hsieh and I.S. Dhillon.
\newblock Fast coordinate descent methods with variable selection for
  non-negative matrix factorization.
\newblock In {\em Proceedings of the 17th ACM SIGKDD international conference
  on Knowledge discovery and data mining}, pages 1064--1072. ACM, 2011.

\bibitem{HSS14}
K.~Huang, N.~Sidiropoulos, and A.~Swami.
\newblock Non-negative matrix factorization revisited: Uniqueness and algorithm
  for symmetric decomposition.
\newblock {\em IEEE Transactions on Signal Processing}, 62(1):211--224, 2014.

\bibitem{kalofolias2012computing}
V.~Kalofolias and E.~Gallopoulos.
\newblock Computing symmetric nonnegative rank factorizations.
\newblock {\em Linear Algebra and its Applications}, 436(2):421--435, 2012.

\bibitem{kim2008toward}
J.~Kim and H.~Park.
\newblock Toward faster nonnegative matrix factorization: A new algorithm and
  comparisons.
\newblock In {\em Data Mining, 2008. ICDM'08. Eighth IEEE International
  Conference on}, pages 353--362. IEEE, 2008.

\bibitem{kim2011fast}
J.~Kim and H.~Park.
\newblock {Fast nonnegative matrix factorization: An active-set-like method and
  comparisons}.
\newblock {\em SIAM J. on Scientific Computing}, 33(6):3261--3281, 2011.

\bibitem{KBV09}
Y.~Koren, R.~Bell, and C.~Volinsky.
\newblock Matrix factorization techniques for recommender systems.
\newblock {\em Computer}, (8):30--37, 2009.

\bibitem{kuang2012symmetric}
D.~Kuang, H.~Park, and C.H.Q. Ding.
\newblock Symmetric nonnegative matrix factorization for graph clustering.
\newblock In {\em SIAM Conf. on Data Mining (SDM)}, volume~12, pages 106--117,
  2012.

\bibitem{kuang2013symnmf}
D.~Kuang, S.~Yun, and H.~Park.
\newblock {SymNMF: nonnegative low-rank approximation of a similarity matrix
  for graph clustering}.
\newblock {\em Journal of Global Optimization}, 62(3):545--574, 2014.

\bibitem{LS99}
D.D. Lee and H.S. Seung.
\newblock {Learning the Parts of Objects by Nonnegative Matrix Factorization}.
\newblock {\em Nature}, 401:788--791, 1999.

\bibitem{LS2}
D.D. Lee and H.S. Seung.
\newblock {Algorithms for Non-negative Matrix Factorization}.
\newblock {\em In Advances in Neural Information Processing}, 13, 2001.

\bibitem{li2009fastnmf}
L.~Li and Y.-J. Zhang.
\newblock {FastNMF: highly efficient monotonic fixed-point nonnegative matrix
  factorization algorithm with good applicability}.
\newblock {\em Journal of Electronic Imaging}, 18(3):033004--033004, 2009.

\bibitem{long2007relational}
Bo~Long, Zhongfei~Mark Zhang, Xiaoyun Wu, and Philip~S Yu.
\newblock Relational clustering by symmetric convex coding.
\newblock In {\em Proceedings of the 24th international conference on Machine
  learning}, pages 569--576. ACM, 2007.

\bibitem{trefethen1997numerical}
L.N. Trefethen and D.~Bau~III.
\newblock {\em Numerical linear algebra}, volume~50.
\newblock SIAM, 1997.

\bibitem{UHZB14}
M.~Udell, C.~Horn, R.~Zadeh, and S.~Boyd.
\newblock Generalized low rank models.
\newblock {\em Foundations and Trends{\textregistered} in Machine Learning},
  2015.
\newblock To appear, arXiv:1410.0342.

\bibitem{VGLZD15}
A.~Vandaele, N.~Gillis, Q.~Lei, K.~Zhong, and I.~Dhillon.
\newblock Coordinate descent methods for symmetric nonnegative matrix
  factorization.
\newblock {arXiv:1509.01404}, 2015.

\bibitem{wright2015coordinate}
S.J. Wright.
\newblock Coordinate descent algorithms.
\newblock {\em Mathematical Programming}, 151(1):3--34, 2015.

\bibitem{yglcw13}
X.~Yan, J.~Guo, S.~Liu, X.~Cheng, and Y.~Wang.
\newblock Learning topics in short texts by non-negative matrix factorization
  on term correlation matrix.
\newblock In {\em Proc. of the SIAM Int. Conf. on Data Mining}. SIAM, 2013.

\bibitem{yang2012clustering}
Z.~Yang, T.~Hao, O.~Dikmen, X.~Chen, and E.~Oja.
\newblock Clustering by nonnegative matrix factorization using graph random
  walk.
\newblock In {\em Advances in Neural Information Processing Systems}, pages
  1079--1087, 2012.

\bibitem{yang2012quadratic}
Zhirong Yang and Erkki Oja.
\newblock Quadratic nonnegative matrix factorization.
\newblock {\em Pattern Recognition}, 45(4):1500--1510, 2012.

\bibitem{zass2005unifying}
Ron Zass and Amnon Shashua.
\newblock A unifying approach to hard and probabilistic clustering.
\newblock In {\em Computer Vision, 2005. ICCV 2005. Tenth IEEE International
  Conference on}, volume~1, pages 294--301. IEEE, 2005.

\bibitem{ZG05}
S.~Zhong and J.~Ghosh.
\newblock Generative model-based document clustering: a comparative study.
\newblock {\em Knowledge and Information Systems}, 8(3):374--384, 2005.

\end{thebibliography}

\end{document}